\documentclass{amsart}

\usepackage{amssymb,tikz}
\usepackage{amscd}
\usepackage{latexsym}
\usepackage{amsfonts}
\usepackage{amsmath}
\usepackage[latin1]{inputenc}

\newtheorem{proposition}{Proposition}
\newtheorem{lemma}{Lemma}
\newtheorem{theorem}{Theorem}
\newtheorem{corollary}{Corollary}
\newtheorem{definition}{Definition}
\newtheorem{remark}{Remark}
\newtheorem{conjecture}{Conjecture}

\newcommand{\eproof}{\begin{flushright} $\square$ \end{flushright}}

\DeclareMathOperator{\supp}{Supp}
\DeclareMathOperator{\End}{End}

\DeclareMathOperator{\tr}{Tr}
\DeclareMathOperator{\dil}{Li_2}

\DeclareMathOperator{\Ob}{Ob}

\DeclareMathOperator{\RAMAN}{\Psi}
\DeclareMathOperator{\FQDILOG }{\phi}

\newcommand{\B}{A}

\newcommand{\cff}{\psi}
\newcommand{\CLA}{c_\la}
\newcommand{\cla}{c_\la}
\newcommand{\COMPLEXS}{\mathbb C}

\newcommand{\IHG}[1]{{\Psi_#1}}
\newcommand{\IMUN}{\mathsf i}
\newcommand{\imun}{\mathsf i}

\newcommand{\INTEGERS}{\mathbb Z}

\newcommand{\la}{\mathsf{b}}

\newcommand{\MOM}{\mathsf p}

\newcommand{\POS}{\mathsf q}
\newcommand{\PTOLEMY}{\mathsf T}

\newcommand{\QDILOG}{\Phi_{\la}}
\newcommand{\QDILOGI}{\bar\Phi_{\la}}

\newcommand{\REALS}{\mathbb R}

\newcommand{\sfu}{\mathsf u}
\newcommand{\sfv}{\mathsf v}

\newcommand{\wv}{\vec{\mathsf w}}

\newcommand{\im}{\mathop{\fam0 Im}\nolimits}

\newcommand{\re}{\mathop{\fam0 Re}\nolimits}

\newcommand{\Hom}{\mathop{\fam0 Hom}\nolimits}

\newcommand{\vol}{\mathop{\fam0 Vol}\nolimits}

\newcommand{\sign}{\mathop{\fam0 sign}\nolimits}

\newcommand{\bC}{{\mathbb C}}
\newcommand{\bR}{{\mathbb R}}

\newcommand{\C}{\mathcal{C}}

\newcommand{\Z}{{\mathbb Z}}
\newcommand{\bZ}{\Z{}}

\newcommand{\D}{{\mathcal D}}
\newcommand{\ra}{\mathop{\fam0 \rightarrow}\nolimits}

\renewcommand{\L}{{\mathcal L}}

\newcommand{\SU}{\mathop{\fam0 SU}\nolimits}

\newcommand{\CC}{\mathcal{B}}

\renewcommand{\S}{\mathcal{S}}

\newcommand{\OBJ}{\mathrm{Obj}}
\newcommand{\mor}{\mathrm{Mor}}
\newcommand{\wf}{\mathrm{WF}}

\newcommand{\myhom}{\mathop{\theta}}
\newcommand{\parfun}{\mathop{Z_\hbar}}
\newcommand{\cont}{\mathsf{E}}

\newcommand{\sA}{\mathsf{A}}
\newcommand{\sB}{\mathsf{B}}

\begin{document}

\title{A TQFT from quantum Teichm\"{u}ller theory}

\author{J{\o}rgen Ellegaard Andersen}
\address{Center for Quantum Geometry of Moduli Spaces\\
        University of Aarhus\\
        DK-8000, Denmark}
\email{andersen@qgm.au.dk}

\author{Rinat Kashaev}
\address{University of Geneva\\
2-4 rue du Li\`evre, Case postale 64\\
 1211 Gen\`eve 4, Switzerland}
\email{rinat.kashaev@unige.ch}

\thanks{Supported in part by the center of excellence grant ``Center for quantum geometry of Moduli Spaces" from the Danish National Research Foundation, and Swiss National Science Foundation}

\begin{abstract}
By using quantum Teichm\"uller theory, we construct a one parameter family of TQFT's on the categroid of admissible leveled shaped 3-manifolds.
\end{abstract}

\maketitle

\section{Introduction}\label{intro}

Topological Quantum Field Theories in dimension $2 +1$ were discovered and axiomatized by Atiyah~\cite{At}, Segal~\cite{S} and Witten~\cite{W}, and first constructed by Reshetikhin and Turaev \cite{RT1,RT2,T}. These TQFT's are constructed by combinatorial means from the finite dimensional representation category of the quantum group $U_q(sl(2,\bC))$, where $q$ is a root of unity, and are
defined on a cobordism category which is a slight extension of the $2+1$ dimensional cobordism category, where the three manifolds in question are framed and their two dimensional boundaries have the so called extended structures, which include the choice of a Lagrangian subspace of the first homology. Reshetikhin and Turaev used surgery on links and Kirby calculus to show that their TQFT's are well defined. In the particular case of 3-manifolds given as mapping cylinders, the Reshetikhin--Turaev TQFT's  produce representations of centrally extended surface mapping class groups on finite dimensional vector spaces \cite{T}. These TQFT's and their associated representations of centrally extended surface mapping class groups were also constructed by purely topological means by Blanchet, Habegger, Masbaum and Vogel in \cite{BHMV1}, \cite{BHMV2}.

Quantum Teichm\"{u}ller theory, as developed by Kashaev \cite{K1}, and Chekhov and Fock \cite{CF}, produces unitary representations of centrally extended mappings class groups of punctured surfaces on infinite-dimensional Hilbert spaces. The central ingredients in this theory are, on the one hand, Penner's cell decomposition of decorated Teichm\"{u}ller space and the associated Ptolemy groupoid \cite{Pen1} with its many applications summaries in \cite{Pen2} and on the other Faddeev's quantum dilogarithm \cite{F} which finds its origins and applications in quantum integrable systems \cite{FKV,BMS1,BMS2,Te}. Faddeev's quantum dilogarithm has already been used  in formal state-integral constructions of perturbative invariants of three manifolds in the works \cite{H1,H2,DGLZ,DFM,D}, but without addressing the  important questions of convergence or independence of the choice of triangulation.

In this paper, we address the question of promoting quantum Teichm\"{u}ller theory to a TQFT. The main obstacle in constructing such TQFT comes, of course, from the fact that the target category cannot be the category of finite dimensional vector spaces, and one has to make an appropriate choice which guarantees that the functor is well defined, in particular, that the relevant integrals always converge. Unfortunately or fortunately, the obvious choice of the category of Hilbert spaces and bounded operators is not enough in our case: we have to go to the framework of the categroid of temperate distributions.  The starting point is the combinatorial setting of triangulated three manifolds. As our construction verifies invariance under only those changes of triangulations which do not remove or add vertices, we are naturally led to consideration of the cobordism category of pseudo 3-manifolds. As a consequence of the fact that distributions cannot always be multiplied, our TQFT is well defined on only  pseudo 3-manifolds with trivial second homology group of the complement of the vertices. In addition,  an extra structure (called shape structure) on our cobordisms is needed which is closely related to the angle structures  on ideal triangulations of hyperbolic manifolds introduced by Casson, Rivin and Lackenby \cite{C,R,L}. More precisely, a shape structure is a certain equivalence class of angle structures. Also, in analogy with the theory of Reshetikhin and Turaev, one more extra structure (called level) is used in order to handle the phase ambiguity of  the theory.  

In the rest of this introduction, we define all these structures and our TQFT in precise terms.

\subsection{Oriented triangulated pseudo $3$-manifolds}
Consider the disjoint union of finitely many copies of the standard $3$-simplices in $\REALS^3$, each having totally ordered vertices. Notice, that the vertex order induces orientations on edges. Identify  some codimension-1 faces of this union
in pairs by vertex order preserving and orientation reversing affine homeomorphisms called \emph{gluing homeomorphisms}. The quotient space $X$ is a specific CW-complex with oriented edges which will be called an oriented \emph{triangulated pseudo $3$-manifold}. For $i\in\{0,1,2,3\}$, we will denote by $\Delta_i(X)$ the set of $i$-dimensional cells in $X$. For any $i>j$, we also denote

\[
\Delta_{i}^{j}(X)=\{(a,b)\vert\ a\in\Delta_i(X),\ b\in\Delta_j(a)\}
\]
with natural projection maps
\[
\phi_{i,j}\colon\Delta_{i}^{j}(X)\to\Delta_{i}(X),\quad \phi^{i,j}\colon\Delta_{i}^{j}(X)\to\Delta_{j}(X).
\]
We also have the canonical partial boundary maps
\[
\partial_i\colon \Delta_{j}(X)\to \Delta_{j-1}(X),\quad 0\le i\le j,
\]
which in the case of a $j$-dimensional simplex $S=[v_0,v_1,\ldots,v_{j}]$ with ordered vertices $v_0,v_1,\ldots,v_{j}$ in $\REALS^3$ take the form
\[
\partial_iS=[v_0,\ldots,v_{i-1},v_{i+1},\ldots,v_j],\quad i\in\{0,\ldots,j\}.
\]
\subsection{Shaped 3-manifolds}
Let $X$ be an oriented triangulated pseudo 3-manifold.

\begin{definition}
A \emph{Shape structure} on $X$  is an assignment  to each edge of each tetrahedron of $X$ a positive number called the \emph{dihedral angle}
$$
\alpha_X\colon \Delta_{3}^{1}(X)\to\bR_{+}
$$
so that
the sum of the three angles at the edges from each vertex of each tetrahedron is $\pi$. An oriented triangulated pseudo 3-manifold with a shape structure will be called a \emph{shaped pseudo 3-manifold}. We denote the set of shape structures on $X$ by $S(X)$.
\end{definition}

It is straightforward to see that dihedral angles at opposite edges of any tetrahedron are equal, so that each tetrahedron acquires three dihedral angles associated to three pairs of opposite edges which sum up to $\pi$. The usual complex shape variable for a tetrahedron $[v_0,v_1,v_2,v_3],$ with dihedral angles $\alpha,\beta,\gamma$ associated to the edges $[v_0,v_1], [v_0,v_2],[v_0,v_3]$, is
$$q = \frac{\sin(\beta)}{\sin(\gamma)} e^{\imun\alpha}.$$

\begin{definition}
To each shape structure on $X$, we associate a \emph{Weight function}
$$
\omega_X : \Delta_1(X) \ra \bR_{+},
$$
which associates to each edge $e$ of $X$ the sum of dihedral angles around it
$$
\omega_X(e)=\sum_{a\in(\phi^{3,1})^{-1}(e)}\alpha_X(a).
$$
\end{definition}

An \emph{angle structure} on a closed triangulated pseudo $3$-manifold $X$, introduced by Casson, Rivin and Lackenby \cite{C,R,L}, is a shape structure whose weight function takes the value $2\pi$ on each edge.
\begin{definition}
An edge $e$ of a shaped pseudo 3-manifold $X$  will be called \emph{balanced} if it is internal and $\omega_X(e)=2\pi$. An edge which is not balanced will be called \emph{unbalanced}. A shaped pseudo 3-manifold will be called \emph{fully balanced} if all edges of $X$ are balanced.
\end{definition}
Thus, a shape structure is an angle structure if it is fully balanced. By definition, a shaped pseudo 3-manifold can be fully balanced only if its boundary is empty.

\subsection{The $\bZ/3\bZ$-action on pairs of opposite edges of tetrahedra}
Let $X$ be an oriented triangulated pseudo 3-manifold.
For each vertex of each tetrahedron of $X$, the orientation induces a cyclic order on the three edges meeting at the vertex. Moreover, this cyclic order induces a cyclic order on the set of pairs of opposite edges of the tetrahedron. We denote by $\Delta_{3}^{1/p}(X)$ the set of pairs of opposite edges of all tetrahedra. Set-theoretically, it is the quotient set of $\Delta_{3}^{1}(X)$ with respect to the equivalence relation given by all pairs of opposite edges of all tetrahedra. We denote by
\[
p\colon \Delta_{3}^{1}(X)\to\Delta_{3}^{1/p}(X)
\]
the corresponding quotient map, and  we define a skew-symmetric function
\[
\varepsilon_{a,b}\in\{0,\pm1\},\quad \varepsilon_{a,b}=-\varepsilon_{b,a},\quad a,b\in\Delta_{3}^{1/p}(X),
\]
with the value $\varepsilon_{a,b}=0$ if the underlying tetrahedra are distinct and the value $\varepsilon_{a,b}=+1$ if the underlying tetrahedra coincide and the pair of opposite edges associated with $b$ is cyclically preceded by that of $a$. Notice that a shape structure on $X$ descents to a positive real valued function on the set $\Delta_{3}^{1/p}(X)$.

\subsection{Leveled shaped 3-manifolds}
In order for our TQFT to be well defined, we need to extend the shape structure by a real parameter. This is the analog of framing number in the context of the Reshetikhin--Turaev TQFT.

\begin{definition}
A \emph{leveled shaped pseudo 3-manifold} is a pair $(X,\ell_X)$ consisting of a shaped pseudo 3-manifold $X$ and a real number $\ell_X\in\bR$ called the \emph{level}. We denote by $LS(X)$ the set of all leveled shaped structures on $X$.
\end{definition}

As will be demonstrated below, our TQFT will be well defined on a certain sub-categroid\footnote{See Appendix 
C for the definition of Categroids.} of the category of leveled shaped pseudo $3$-manifolds. Moreover, it will enjoy a certain gauge-invariance we will now describe.

\begin{definition}\label{gaugeeq}
Two leveled shaped pseudo 3-manifolds $(X,\ell_X)$ and $(Y,\ell_Y)$ are called \emph{gauge equivalent} if there exist an isomorphism $h\colon X\to Y$ of the underlying cellular structures and a function
\[
g\colon \Delta_1(X)\to \REALS
\]
such that
\[
\Delta_1(\partial X)\subset g^{-1}(0),
\]
\[
\alpha_Y(h(a))=\alpha_X(a)+\pi\sum_{b\in\Delta_{3}^{1}(X)} \varepsilon_{p(a),p(b)}g(\phi^{3,1}(b)),\quad \forall a\in\Delta_{3}^{1}(X),
\]
and
\[
\ell_Y=\ell_X+\sum_{e\in\Delta_1(X)}g(e)\sum_{a\in(\phi^{3,1})^{-1}(e)}
\left(\frac13-\frac{\alpha(a)}\pi\right).
\]
\end{definition}

It is easily seen that the weights on edges are gauge invariant in the sense that
\[
\omega_{X}=\omega_Y\circ h.
\]

\begin{definition}
Two leveled shape structures $(\alpha_X,\ell_X)$ and $(\alpha'_X,\ell'_X)$ on a pseudo $3$-manifold $X$ are called \emph{based gauge equivalent} if they are gauge equivalent in the sense of Definition \ref{gaugeeq}, where isomorphism $h:X \ra X$ in question is the identity.
\end{definition}

We observe that the (based) gauge equivalence relation on leveled shaped pseudo $3$-manifolds induces a (based) gauge equivalence relation on shaped pseudo $3$-manifolds under the map which forgets the level. Let $N_0(X)$ be a sufficiently small tubular neighborhood of $\Delta_0(X)$. Then $\partial N_0(X)$ is a two dimensional surface, which is possibly disconnected and possibly with boundary, if $\partial X \neq \varnothing$. Let the set of gauge equivalence classes of based leveled shape structures on $X$ be denoted $LS_r(X)$ and let $S_r(X)$ denote the set of gauge equivalence classes of based shape structures on $X$.

Let us now describe the based gauge equivalence classes of (leveled) shape structures on a pseudo $3$-manifold $X$. In order to do this we will need the notion of generalized shape structure.

\begin{definition}\label{genshape}
A \emph{generalized shape} structure on $X$ is an assignment  of a real number to each edge of each tetrahedron, so that
the sum of three numbers at the edges from each vertex of each tetrahedron is $\pi$. Leveled generalized shaped structures as well as their gauge equivalence are defined analogously to leveled shaped structures and their gauge equivalence. The space of based gauge equivalence classes of generalized shape structures will be denoted $\tilde{S}_r(X)$ and the space of based leveled generalized shape structures is denoted $\widetilde{LS}_r(X)$.
\end{definition}

We observe that $S_r(X)$ is an open convex subset of $\tilde{S}_r(X)$.

The map which assigns  to a generalized shape structure $\alpha_X\in \tilde{S}(X)$ the corresponding weight function $\omega_X : \Delta_1(X) \ra \bR$ is denoted
$$
\tilde{\Omega}_X : \tilde{S}(X) \ra \bR^{\Delta_1(X)}.
$$
Since this map is gauge invariant it induces a unique map
$$
\tilde{\Omega}_{X,r} : \tilde{S}_r(X) \ra \bR^{\Delta_1(X)}.
$$

\begin{theorem}\label{Shape}
The map
$$
\tilde{\Omega}_{X,r} : \tilde{S}_r(X) \ra \bR^{\Delta_1(X)},
$$
 is an affine $H^1(\partial N_0(X), \bR)$-bundle.
The space $\tilde{S}_r(X)$ carries a Poisson structure whose symplectic leaves are the fibers of $\tilde{\Omega}_{X,r}$ and which is identical to the Poisson structure induced by the $H^1(\partial N_0(X), \bR)$-bundle structure. The natural projection map from $\widetilde{LS}_r(X) $ to $\tilde{S}_r(X)$ is an affine  $\bR$-bundle which restricts to the affine $\bR$-bundle $LS_r(X)$ over $S_r(X)$.

If $h:X \ra Y$ is an isomorphism of cellular structures, then we get an induced Poisson isomorphism $h^* : \tilde{S}_r(Y) \ra \tilde{S}_r(X)$ which is an affine bundle isomorphism with respect to the induced group homomorphism
\[
h^* : H^1(\partial N_0(Y), \bR)\ra H^1(\partial N_0(X), \bR)
 \]
 and which maps $S_r(Y)$ to $S_r(X)$. Furthermore, $h$ induces an isomorphism
\[
h^*\colon \widetilde{LS}_r(Y) \to \widetilde{LS}_r(X)
 \]
 of affine $\bR$-bundles covering the map
 \(
 h^*\colon \tilde{S}_r(Y) \to \tilde{S}_r(X)
 \)
  which also maps $LS_r(Y)$ to $LS_r(X)$.

\end{theorem}

This theorem will be proved in Section~\ref{reduction}, where we will explain how $\tilde{S}_r(X)$ arises as a symplectic reduction of the space of all generalized shape structures and thus carries a Poisson structure whose symplectic leaves are the fibers of $\tilde{\Omega}_X$ which is identical to the Poisson structure induced by the $H^1(\partial N_0(X), \bR)$-bundle structure.

\subsection{The $3-2$ Pachner moves}\label{3-2PM}
Let $X$ be a shaped pseudo 3-manifold. Let $e$ be a balanced edge of $X$ shared by exactly three distinct tetrahedra $t_1,t_2,t_3$. Let $S$ be a shaped pseudo 3-submanifold of $X$ composed of the tetrahedra $t_1,t_2,t_3$. Note that $S$ has $e$ as its only internal and balanced edge. There exists another triangulation $S_e$ of the topological space underlying $S$ such that the triangulation of $\partial S$ coincides with that of $\partial S_e$, but which consists of only two tetrahedra $t_4,t_5$. It is obtained by removing the edge $e$ so that $\Delta_1(S_e)=\Delta_1(S)\setminus\{e\}$. Moreover, there exists a unique shaped structure on $S_e$ which induces the same weights as the shape structure of $S$. For some choices of shape variables $(\alpha_i,\beta_i,\gamma_i)$ for $t_i$ (where $\alpha_i$ are the angles at $e$), the explicit map is given by

\begin{equation}\label{P32a}\begin{array}{ll}
\alpha_4 = \beta_2 + \gamma_1 & \alpha_5 = \beta_1 + \gamma_2\\
\beta_4 = \beta_1 + \gamma_3 & \beta_5 = \beta_3 + \gamma_1\\
\gamma_4 = \beta_3 + \gamma_2 & \gamma_5 = \beta_2 + \gamma_3.
\end{array}
\end{equation}
We observe that the equation $\alpha_1+\alpha_2 + \alpha_3 = 2\pi$ guaranties that the sum of the angles for $t_4$ and $t_5$ sums to $\pi$. Moreover it is clear from these equations that the positivity of the angles for $t_1, t_2, t_3$ guaranties that the angles for $t_4$ and $t_5$ are positive. Conversely we see that it is not automatic that we can solve for (positive) angles for $t_1,t_2,t_3$ given the angles for $t_4$ and $t_5$. However if we have two positive solutions for the angles for $t_1,t_2,t_3$, then they are gauge equivalent and satisfies that $\alpha_1+\alpha_2 + \alpha_3 = 2\pi$.

\begin{definition}
We say that a shaped pseudo 3-manifold $Y$ is obtained from $X$ by a \emph{shaped $3-2$ Pachner move} along $e$ if $Y$ is obtained from $X$ by replacing $S$ by $S_e$, and we write $Y=X_e$.
\end{definition}

We observe from the above that there is a canonical map from the set of shape structures on $X$ to the set of shape structures on $Y$:
$$ P^e : S(X) \ra  S(Y).$$
This map naturally extends to a map from all generalized shape structures on $X$ to the set of generalized shape structures on $Y$:
$$ \tilde{P}^e : \tilde S(X) \ra  \tilde S(Y).$$
We get the following commutative diagram
$$\begin{CD}
\tilde \Omega_X(e)^{-1}(2\pi) @>{\tilde P}>> \tilde S(Y)\\
@VVV @VVV\\
\tilde \Omega_{X, r} (e)^{-1}(2\pi)  @>\tilde P_r>> \tilde S_r(Y)\\
@VV{\rm{proj}\ \circ \ \tilde \Omega_{X,r}}V @VV{\tilde \Omega_{Y,r}}V\\
\bR^{\Delta_1(X)-e} @> = >>  \bR^{\Delta_1(Y)}
\end{CD}
$$
Moreover
$$\tilde P_r( \tilde \Omega_{X, r} (e)^{-1}(2\pi)  \cap S_r(x)) \subset S_r(Y).$$

\begin{theorem}\label{3-2SS}
Suppose that  a shaped pseudo 3-manifold $Y$ is obtained from a shaped pseudo 3-manifold $X$ by a leveled shaped $3-2$ Pachner move.
Then the map $\tilde P_r$ is a Poisson isomorphism, which is covered by an affine $\bR$-bundle isomorphism from $\widetilde{LS}_r(X)|_{\tilde \Omega_{X, r} (e)^{-1}(2\pi) }$ to $\widetilde{LS}_r(Y)$.
\end{theorem}

See Section~\ref{reduction} for the proof.

We also say that a leveled shaped pseudo 3-manifold $(Y,\ell_Y)$ is obtained from a leveled shaped pseudo 3-manifold $(X,\ell_X)$ by a \emph{leveled shaped $3-2$ Pachner move} if there exists $e\in\Delta_1(X)$ such that $Y=X_e$ and
\[
\ell_Y=\ell_X +\frac1{12\pi}\sum_{a\in(\phi^{3,1})^{-1}(e)}
\sum_{b\in\Delta_{3}^{1}(X)}\varepsilon_{p(a),p(b)}\alpha_X(b).
\]

\begin{definition}
A (leveled) shaped pseudo 3-manifold $X$ is called a \emph{Pachner refinement} of a (leveled) shaped pseudo 3-manifold $Y$ if there exists a finite sequence of (leveled) shaped pseudo 3-manifolds
\[
X=X_1,X_2,\ldots,X_n=Y
\]
such that for any $i\in\{1,\ldots,n-1\}$, $X_{i+1}$ is obtained from $X_i$ by a (leveled) shaped $3 - 2$ Pachner move. Two (leveled) shaped pseudo 3-manifolds $X$ and $Y$ are called \emph{equivalent} if there exist gauge equivalent (leveled) shaped pseudo 3-manifolds $X'$ and $Y'$ which are respective  Pachner refinements of $X$ and $Y$.
\end{definition}

\begin{theorem}\label{equivalencetheo}
Suppose two (leveled) shaped pseudo $3$-manifolds $X$ and $Y$ are equivalent. Then there exist $D\subset \Delta_1(X)$ and $D' \subset \Delta_1(Y)$ and a bijection
$$i : \Delta_1(X) - D \ra  \Delta_1(Y) - D'$$
and a Poisson isomorphism
$$R :  \tilde \Omega_{X, r} (D)^{-1}(2\pi) \ra \tilde \Omega_{Y, r} (D')^{-1}(2\pi),$$
which is covered by  an affine $\bR$-bundle isomorphism from $\widetilde{LS}_r(X)|_{\tilde \Omega_{X, r} (D)^{-1}(2\pi) }$ to $\widetilde{LS}_r(Y)|_{\tilde \Omega_{Y, r} (D')^{-1}(2\pi)}$ and such that we get the following commutative diagram
 $$\begin{CD}
\tilde \Omega_{X, r} (D)^{-1}(2\pi)  @>R>>  \tilde \Omega_{Y, r} (D')^{-1}(2\pi)\\
@VV{\rm{proj}\ \circ \ \tilde \Omega_{X,r}}V @VV{\rm{proj}\ \circ \ \tilde \Omega_{Y,r}}V\\
\bR^{\Delta_1(X)-D} @> i^* >>  \bR^{\Delta_1(Y)-D'}.
\end{CD}
$$
Moreover the isomorphism $R$ takes an open convex subset $U$ of $S_r(X)\cap \tilde \Omega_{X, r} (D)^{-1}(2\pi) $ onto an open convex subset $U'$ of  $S_r(Y)\cap \tilde \Omega_{Y, r} (D)^{-1}(2\pi) $. If $S_r(X)$ (equivalently $S_r(Y)$) is non empty, then so are $U$ and $U'$.
\end{theorem}

This Theorem is an immediate consequence of Theorem \ref{Shape} and Theorem \ref{3-2SS}.

 \subsection{The categroid of admissible leveled shaped pseudo $3$-manifolds}

Equivalence classes of leveled shaped pseudo 3-manifolds form morphisms of a cobordism category $\CC$, where the objects are triangulated surfaces, and composition is gluing along the relevant parts of the boundary by edge orientation preserving and face orientation reversing CW-homeomorphisms with the obvious composition of dihedral angles and addition of levels. Depending on the way of splitting the boundary, one and the same leveled shaped pseudo 3-manifold can be interpreted as different morphisms in $\CC$. Nonetheless, there is one canonical choice defined as follows.

For a tetrahedron $T=[v_0,v_1,v_2,v_3]$ in $\REALS^3$ with ordered vertices $v_0,v_1,v_2,v_3$, we define its sign
\[
\sign(T)=\sign(\det(v_1-v_0,v_2-v_0,v_3-v_0)),
\]
as well as the signs of its faces
\[
\sign(\partial_iT)=(-1)^{i}\sign(T),\quad i\in\{0,\ldots,3\}.
\]
For a pseudo 3-manifold $X$, the signs of faces of  the tetrahedra of $X$ induce a sign function on the faces of the boundary of $X$,
\[
\sign_X\colon \Delta_2(\partial X)\to\{\pm1\},
\]
which permits to split the boundary of $X$ into two components,
\[
\partial X=\partial_+X\cup\partial_-X,\quad
\Delta_2(\partial_\pm X)=\sign_X^{-1}(\pm1),
\]
composed of equal number of triangles. For example, in the case of a tetrahedron $T$ with $\sign(T)=1$, we have $\Delta_2(\partial_+T)=\{\partial_0T,\partial_2T\}$, and $\Delta_2(\partial_-T)=\{\partial_1T,\partial_3T\}$.
In what follows, unless specified otherwise, (the equivalence class of) a leveled shaped pseudo 3-manifold $X$  will always be thought of as a $\CC$-morphism between the objects $\partial_-X$ and $\partial_+X$, i.e.
\[
X\in\Hom_\CC(\partial_-X,\partial_+X).
\]
Our TQFT is not defined on the full category $\CC$, but only on a certain sub-categroid we will now define.

\begin{definition}
An oriented triangulated pseudo $3$-manifold is called \emph{admissible} if
$$
H_2(X-\Delta_0(X), \bZ) = 0.
$$
\end{definition}

\begin{definition}
The categroid $\CC_a$ of admissible leveled shaped pseudo $3$-manifolds is the sub-categroid of the category of leveled shaped pseudo $3$-manifolds whose morphisms consist of equivalence classes of admissible leveled shaped pseudo $3$-manifolds.
\end{definition}

Gluing in this sub-categroid is the one induced from the category $\CC$ and it is only defined for those pairs of admissible morphisms for which the glued morphism in $\CC$ is also admissible.

\subsection{The TQFT functor}
The main result of this paper is the construction of a one parameter family of TQFT's on admissible leveled shaped pseudo 3-manifolds, i.e. a family of functor's $\{F_\hbar\}_{\hbar\in\REALS_{+}}$ from the cobordism categroid $\CC_a$ of admissible leveled shaped pseudo 3-manifolds to the categroid of temperate distributions $\D$, which we will now describe.

Recall that the space of (complex) temperate distributions $\S'(\bR^n)$ is the space of continuous linear functionals on the (complex) Schwartz space $\S(\bR^n)$. By the Schwartz presentation theorem (see e.g. Theorem V.10 p. 139 \cite{RS1}), any temperate distribution can be represented by a finite derivative of a continuous function with polynomial growth, hence we may think of temperate distributions as functions defined on $\bR^n$. We will use the notation $\varphi(x)\equiv\langle x\vert\varphi\rangle$ for any $\varphi\in \S'(\bR^n)$ and $x\in \bR^n$. This notation should be considered in the usual distributional sense, e.g.
$$
\varphi(f) = \int_{\bR^n} \varphi(x) f(x) dx.
$$ This formula further exhibits the inclusion $\S(\bR^n)\subset \S'(\bR^n)$.

\begin{definition}
The categroid $\D$ has as objects finite sets and for two finite sets $n,m$ the set of morphisms from $n$ to $m$ is
$$
\Hom_{\D}(n,m) = \S'(\bR^{n\sqcup m}).
$$
\end{definition}
Denoting by  $ \L( \S(\bR^{n}), \S'(\bR^{m}))$ the space of continuous linear maps from $\S(\bR^{n})$ to $ \S'(\bR^{m})$, we remark that we have an isomorphism
\begin{equation}\label{iso}
\tilde{\cdot} :   \L( \S(\bR^{n}), \S'(\bR^{m})) \ra \S'(\bR^{n\sqcup m})
\end{equation}
determined by the formula
$$
\varphi(f)(g) = \tilde{\varphi}(f\otimes g)
$$
for all $\varphi \in  \L( \S(\bR^{n}), \S'(\bR^{m}))$, $f\in  \S(\bR^{n})$, and $g\in \S(\bR^{m})$. This is the content of the Nuclear theorem, see e.g. \cite{RS1}, Theorem V.12, p. 141. The reason why we get a categroid rather than a category, comes from the fact that we cannot compose all composable (in the usual categorical sense) morphisms, but only a subset thereof.

The partially defined composition in this categroid is  defined as follows. Let $n,m,l$ be three finite sets and $A\in \Hom_{\D}(n,m)$ and $B\in \Hom_{\D}(m,l)$. According to the temperate distribution analog of Theorem 6.1.2. in \cite{Hor1}, we have pull back maps
$$
\pi_{n,m}^* : \S'(\bR^{n\sqcup m}) \ra \S'(\bR^{n\sqcup m \sqcup l}) \mbox{ and } \pi_{m,l}^* : \S'(\bR^{m\sqcup l}) \ra \S'(\bR^{n\sqcup m \sqcup l}).
$$
By theorem IX.45 in \cite{RS2} (see also Appendix~B), the product
$$
\pi_{n,m}^*(A)\pi_{m,l}^*(B) \in \S'(\bR^{n\sqcup m \sqcup l})
$$
is well defined provided the wave front sets of $\pi_{n,m}^*(A)$ and $\pi_{m,l}^*(B)$ satisfy the following transversality condition
\begin{equation}\label{wftrans}
(\wf(\pi_{n,m}^*(A)) \oplus \wf(\pi_{m,l}^*(B)) )\cap Z_{n\sqcup m \sqcup l} = \varnothing
\end{equation}
where $Z_{n\sqcup m \sqcup l}$ is the zero section of $T^*(\bR^{n\sqcup m \sqcup l})$. If we now further assume that $\pi_{n,m}^*(A)\pi_{m,l}^*(B)$ continuously extends to $\S(\bR^{n\sqcup m \sqcup l})_m$ as is defined in Appendix~B,
then we obtain a well defined element
$$
(\pi_{n,l})_*(\pi_{n,m}^*(A)\pi_{m,l}^*(B) ) \in \S'(\bR^{n\sqcup l}).
$$
\begin{definition}
For $A\in \Hom_{\D}(n,m)$ and $B\in \Hom_{\D}(m,l)$ satisfying condition~(\ref{wftrans}) and such that $\pi_{n,m}^*(A)\pi_{m,l}^*(B)$ continuously extends to $S(\bR^{n\sqcup m \sqcup l})_m$, we define
$$
AB = (\pi_{n,l})_*(\pi_{n,m}^*(A)\pi_{m,l}^*(B) ) \in \Hom_{\D}(n,l).
$$
\end{definition}
For any $A\in  \L( \S(\bR^{n}), \S'(\bR^{m}))$, we have unique adjoint $A^*\in \L( \S(\bR^{m}), \S'(\bR^{n}))$ defined by the formula
$$
A^*(f)(g) = \overline{\bar{f}(A(\bar{g}))}
$$ for all $f\in  \S(\bR^{m})$ and all $g\in \S(\bR^{n})$.

\begin{definition}
A functor $F\colon \CC_a \to \D$ is said to be a  \emph{$*$-functor} if
\[
F(X^*)=F(X)^*,
\]
where $X^*$ is $X$ with opposite orientation, and $F(X)^*$ is the dual map of $F(X)$.
\end{definition}

The central essential ingredient in the construction of our functor is Faddeev's quantum dilogarithm \cite{F}.

\begin{definition} \emph{Faddeev's quantum dilogarithm} is a function of two complex arguments $z$ and $\la$ defined by the formula
\[
\QDILOG(z):=\exp\left(
\int_{C}
\frac{e^{-2\IMUN zw}\, dw}{4\sinh(w\la)
\sinh(w/\la) w}\right),
\]
where the contour $C$ runs along the real axis, deviating into the upper half plane in the vicinity of the origin.
\end{definition}

It is easily seen that  $\QDILOG(z)$ depends on $\la$ only through the combination $\hbar$ defined by the formula
\[
\hbar:=\left(\la+\la^{-1}\right)^{-2}.
\]

We can now state our main theorem.
\begin{theorem}\label{Main}
For any $\hbar\in\REALS_{+}$, there exists a unique $*$-functor
$F_\hbar\colon \CC_a \to \D$ such that
\(
F_\hbar(A)=\Delta_2(A),\  \forall A\in\Ob\CC_a,
\)
and for any admissible leveled shaped  pseudo 3-manifold $(X,\ell_X)$, the associated morphism in $\D$ takes the form
\begin{equation}\label{fh}
F_\hbar(X,\ell_X)=\parfun(X)e^{\imun\pi\frac{\ell_X}{4\hbar}}\in\S'\left(\bR^{\Delta_2(\partial X)}\right),
\end{equation}
where
\(
\parfun(X)
\)
is such that for one tetrahedron $T$ with $\sign(T)=1$, it is given by the formula
\begin{equation}\label{tet-int-f}
\parfun(T)(x)=\delta(x_0+x_2-x_1)\frac{\exp\left(2\pi\imun (x_3-x_2)\left(x_0+\frac{\alpha_3}{2\imun\sqrt{\hbar}}\right)+\pi \imun\frac{\varphi_T}{4\hbar}\right)}{\QDILOG\left(x_3-x_2+\frac{1-\alpha_1}{2\imun\sqrt{\hbar}}\right)}
\end{equation}
where $\delta(t)$ is Dirac's delta-function,
\[
\varphi_T:=\alpha_0\alpha_2+\frac{\alpha_0-\alpha_2}3-\frac{2\hbar+1}{6},\quad
\alpha_i:=\frac1\pi\alpha_T(\partial_i\partial_0T),\quad i\in\{0,1,2\},
\]
and
\[
x_i:=x(\partial_i(T)),\quad x\colon \Delta_2(\partial T)\to\REALS.
\]
\end{theorem}
The main constituents of the proof of this theorem are presented in Sections~\ref{CTO} to \ref{Conv}. The key idea behind it is to use the charged tetrahedral operator $\parfun(T)$ given by formula~\eqref{tet-int-f} and which is further discussed in Section~\ref{CTO}. This operator carries all the necessary symmetries and satisfies the pentagon relation as demonstrated in Sections~\ref{CPI} and \ref{tqft-rules}. The gauge transformation properties of the partition function is established in Section~\ref{GTP}. A certain geometric constraint on the partition function is established in Section~\ref{GC}. The convergence properties under gluing of tetrahedra is proved in Section~\ref{Conv}. We end Section~\ref{Conv} by summarizing the proof of Theorem~\ref{Main}.

\begin{remark}
We emphasize that for an admissible pseudo $3$-manifold $X$, our TQFT functor provides us with the following well defined function
$$
F_\hbar : LS_r(X) \ra \S'(\bR^{\partial X}).
$$
For the case $\partial X = \varnothing$, we have $\S'(\bR^{\partial X})= \bC$ and so, in this case, we simply get a complex valued function on $LS_r(X)$.
In particular, the value of the functor $F_\hbar$ on any fully balanced admissible leveled shaped 3-manifold is a complex number, which is a topological invariant.
\end{remark}

\subsection{Invariants of knots in 3-manifolds}
 By considering ideal triangulations of complements of hyperbolic knots in compact oriented closed 3-manifolds, we obtain knot invariants which are direct analogues of Baseilhac--Benedetti invariants \cite{BB}.  For such an $X$, our invariant is a complex valued function on the affine $\bR$ bundle $LS_r(X)$ over $S_r(X)$, which forms an open convex (if non-empty) subset of the affine space $\tilde{S}_r(X)$, and which is modeled on the real cohomology of the boundary of a tubular neighborhood of the knot. One can study these invariants also in the case of non-hyperbolic knots  whose complements admit ideal triangulations with non-negative angle structures. In this case, one first calculates the partition function for a not fully balance shape structure, and then tries to take a limit to a fully balanced non-negative shape structure. if such limit exists, then this will be the value of the invariant. However, in the case of the unknot in 3-sphere, where the complement is a solid torus, the invariant cannot be calculated.  The reason is that if  a manifold $M$ admits  an ideal triangulation supporting a non-negative angle structure, then according to Casson and Lackenby,  the simple normal surface theory
and the Gauss--Bonnet theorem for angle structures imply that the boundary of $M$ is incompressible in $M$, see \cite{Lackenby}. Thus, $M$ cannot be a solid torus which has compressible boundary\footnote{We thank Feng Luo for explaining to us this point and pointing to ref. \cite{Lackenby}.}.

Another possibility is to consider one-vertex Hamiltonian triangulations (or H-triangulations) of pairs (a compact closed 3-manifold $M$, a knot $K$ in $M$), i.e. one vertex triangulations of $M$ where the knot $K$ is represented by one edge, with degenerate shape structures, where the weight on the knot approaches zero and simultaneously the weights on all other edges approach the balanced value $2\pi$ (assuming that such configurations can be approached via shape structures on a given triangulation). This limit by itself is divergent as a simple pole (after analytic continuation to complex angles) in the weight of the knot, but the residue at this pole is a knot invariant which is the direct analogue of Kashaev's invariants \cite{K4} which are specialisations of the the colored Jones polynomials \cite{MM} and which were at the origin of the hyperbolic volume conjecture \cite{K6}. In the next subsection, we suggest a conjectural relationship between these two types of invariants.

\subsection{Future perspectives}
In this subsection we present a conjecture about our functor $F_\hbar$, which, among other things, provides a relation to the hyperbolic volume in the asymptotic limit $\hbar\to0$. So far, we have been able to check this conjecture for the first two hyperbolic knots.

\begin{conjecture}\label{conj}
Let $M$ be a closed oriented compact 3-manifold.
For any hyperbolic knot $K\subset  M$, there exists a smooth function
$J_{M,K}(\hbar,x)$ on $\REALS_{>0}\times\REALS$  which has the following properties.
\begin{enumerate}
\item
For any fully balanced shaped ideal triangulation $X$ of the complement  of $K$ in $M$, there exist a gauge invariant real  linear combination of dihedral angles $\lambda$, a (gauge non-invariant) real quadratic polynomial of dihedral angles $\phi$ such that
\[
\parfun(X)=e^{\imun\frac{\phi}{\hbar}}\int_{\REALS} J_{M,K}(\hbar,x)e^{-\frac{x\lambda}{\sqrt{\hbar}}} dx
\]
\item
For any one vertex shaped H-triangulation $Y$ of the pair $(M,K)$ there exists  a real quadratic polynomial of dihedral angles $\varphi$ such that
\[
\lim_{\omega_Y\to \tau}\QDILOG\left(\frac{\pi-\omega_Y(K)}{2\pi\imun\sqrt{\hbar}}\right)\parfun(Y)=
e^{\imun\frac{\varphi}{\hbar}-\imun\frac{ \pi}{12}}J_{M,K}(\hbar,0),
\]
where $\tau\colon \Delta_1(Y)\to \REALS$ takes the value $0$ on the knot $K$ and the value $2\pi$ on all other edges.
\item The hyperbolic volume of the complement of $K$ in $M$ is recovered as the following limit:
\[
\lim_{\hbar\to0}2\pi\hbar\log\vert J_{M,K}(\hbar,0)\vert=-\vol(M\setminus K).
\]
\end{enumerate}
\end{conjecture}
\begin{remark}
In part (3) of the conjecture, we have a negative sign in the right hand side which differs from the usual volume conjecture \cite{K6} so that, in this case, the invariant exponentially decays rather than grows, the decay rate being given by the hyperbolic volume.
\end{remark}
\begin{theorem}\label{th:4-1--5-2}
Conjecture \ref{conj} is true for the pairs $(S^3,4_1)$ and $(S^3,5_2)$ with
\[
J_{S^3,4_1}(\hbar,x)=\chi_{4_1}(x),\quad J_{S^3,5_2}(\hbar,x)=\chi_{5_2}(x),
\]
where functions $\chi_{4_1}(x)$ and $\chi_{5_2}(x)$ are defined in \eqref{eq:chi}
\end{theorem}
This theorem is proved in Section~\ref{proof-th3}.

To conclude this introduction,  let us mention that it would be interesting to understand the operators constructed in this paper from the viewpoint of Toeplitz operator constructions of \cite{A} in the context of Reshetikhin-Turaev TQFT.  Besides, Teschner's modular functor \cite{tesch1} derived from quantum Teichm\"{u}ller theory could possibly be behind another formulation of our TQFT.

\subsection*{Acknowledgements} We would like to thank Ludwig Faddeev, Gregor Masbaum, Nikolai Reshetikhin, and Vladimir Turaev  for valuable discussions. Our special thanks go to Feng Luo for  explaining to us the topological significance of non-negative angle structures.

\section{The symplectic space of generalized shape structures}
\label{reduction}
Let $X$ be a pseudo $3$-manifold. First, we recall the Neumann--Zagier symplectic structure $\omega$ on the affine space $\tilde{S}(X)$ \cite{NZ}.
Each $a\in \Delta^{1/p}_3(X)$ induces a function on $\tilde{S}(X)$ which we  denote by $\alpha_a$.
\begin{definition}
The Neumann--Zagier symplectic structure $\omega$ is the unique symplectic structure on $\tilde{S}(X)$ whose induced Poisson bracket $\{\cdot,\cdot\}$ satisfies the equation
$$
\{\alpha_a,\alpha_b\} = \epsilon_{a,b}
$$
for all $a,b\in \Delta^{1/p}_3(X)$.
\end{definition}
We have the following symplectic product decomposition over tetrahedra
$$
\tilde{S}(X) = \prod_{T\in \Delta_3(X)} \tilde{S}(T)
$$
where $\tilde{S}(T)$ is an affine copy of $\bR^2$. The standard symplectic structure on  $\bR^2$ induces a symplectic structure on $\tilde{S}(T)$ which coincides with the Neumann--Zagier structure.

We now define an action of $\bR^{\Delta_1(X)}$ on $\tilde{S}(X)$ by the formulae in Definition \ref{gaugeeq}, where we take $h$ to be the identity map.

\begin{theorem}
The action of $\bR^{\Delta_1(X)}$ on $\tilde{S}(X)$ is symplectic and the map $\tilde{\Omega}_X$ is a moment map for this action.

\end{theorem}

\proof
The fact that $\bR^{\Delta_1(X)}$ acts symplectically on $\tilde{S}(X)$ follows from the fact that $\bR^{\Delta_1(X)}$ acts by translations. Let $v_e\in \mbox{Lie}(\bR^{\Delta_1(X)})$, $e\in \Delta_1(X)$, be the natural basis of $\mbox{Lie}(\bR^{\Delta_1(X)})$. Fixing $e\in \Delta_1(X)$, the basis vector $v_e$ induces the following vector field $X_e$ on $\tilde{S}(X)$
$$
X_e = \pi  \sum_{a\in \Delta^{1/p}_3(X)}   \sum_{b\in (\phi^{3,1})^{-1}(e)}\epsilon_{a,b} \frac{\partial}{\partial \alpha_a}.
$$
By contracting this vector field with the symplectic form, one gets a one form $\Lambda_e$ given by
$$
\Lambda_e = \pi \sum_{a,c\in \Delta^{1/p}_3(X)} \sum_{ b\in (\phi^{3,1})^{-1}(e)} \epsilon_{a, b}\epsilon_{a,c} d\alpha_c.
$$
Now, we compute the exterior derivative of $\tilde{\Omega}_X(e) = \omega_X(e)$
$$
d\omega_X(e) = \sum_{c\in (\phi^{3,1})^{-1}(e)} d\alpha_c.
$$
By computing the components of the one form $\Lambda_e$ in $\tilde{S}(T)$, for each $T\in \Delta_3(X)$, we see that
$$
\Lambda_e =-d\omega_X(e).$$
\eproof
Thus, we see that
$$
\tilde{S}_r(X) = \tilde{S}(X)/\bR^{\Delta_1(X)}
$$
caries a natural Poisson structure induced from the symplectic structure on $\tilde{S}(X)$. Moreover, the moment map $\tilde{\Omega}_X$ is invariant under the action of the group $\bR^{\Delta_1(X)}$ and hence it descends to a well defined map on $\tilde{S}_r(X)$, which we denote $\tilde{\Omega}_{X,r}$ in the introduction.  By the previous theorem, the symplectic leaves of the Poisson structure on
$\tilde{S}_r(X)$ are  the fibers of $\tilde{\Omega}_{X,r}$.

\vskip.3cm

\noindent\emph{Proof of Theorem \ref{Shape}.}
Consider any $m\in \bR^{\Delta_1(X)}$. It is clear that $\tilde{\Omega}_X^{-1}(m) \neq \emptyset$, so pick a point $a\in \tilde{\Omega}_{X,r}^{-1}(m) $. We will now construct a map
$$
H_a : \tilde{\Omega}_{X,r}^{-1}(m)  \ra H^1(\partial N_0(X), \bR)
$$ as follows.
For each $a'\in \tilde{\Omega}_{X,r}^{-1}(m) $, we consider $a-a'$ which we represent by a map from $\Delta^{1/p}_3(X)$ to the reals which satisfies the condition that  the sum in each tetrahedron vanishes. Let $\gamma$ be a closed curve on $\partial N_0(X)$, which consists entirely of oriented curve segments which are normal curves with respect to the induced triangulation on $\partial N_0(X)$. Each oriented normal segment of $\gamma$ gets assigned the corresponding value  $\pm(a-a')$, where the sign is determined by the sign of the segment. The sign on the segment is determined by the sign by which the oriented segment traverses the corresponding wedge of the relevant triangle in the triangulation of $\partial N_0(X)$. We now define $H_a(a')(\gamma)$ to be the sum of these real numbers over all segments of $\gamma$.  We observe that  since  both $a$ and $a'$ are contained in $\tilde{\Omega}_{X,r}^{-1}(m)$, then  $H_a(a')(\gamma)$ only depends on the homology class of $\gamma$ in $H_1(\partial N_0(X), \bR)$, hence
$H_a(a')\in H^1(\partial N_0(X), \bR)$ is well defined. Now, we observe that if we change the representative $a-a'$ by changing the representative of either $a$ or $a'$ we do not change $H_a(a')$.
We observe that there is a natural linear structure on $ \tilde{\Omega}_{X,r}^{-1}(m)$ based at $a$. The map $H_a$ is linear with respect to this linear structure. We now consider the dual map
$$H_a^* :  H_1(\partial N_0(X), \bR) \ra \tilde{\Omega}_{X,r}^{-1}(m)^*.
$$
Again, if $\gamma$ is a smooth curve on $\partial N_0(X)$, then $H_a^*(\gamma)$ is a linear function on $ \tilde{\Omega}_{X,r}^{-1}(m)$. We claim that
$$
\{H_a^*(\gamma_1),H_a^*(\gamma_2)\} = \gamma_1\cdot \gamma_2
$$
for all pairs $\gamma_1,\gamma_2$ of closed curves on $\partial N_0(X)$.
This is an easy straightforward check by taking into account the fact that the curves $\gamma_1$ and $\gamma_2$ can be deformed so that each intersection point becomes the midpoint of an edge shared by two triangles as in this picture
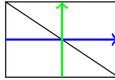
\[
 \begin{tikzpicture}[scale=.5]
 \draw (0,0)rectangle(3,2) (0,2)--(3,0);
 \draw[blue,thick,->] (0,1)--(3,1);
 \draw[green,thick,->] (1.5,0)--(1.5,2);
 \end{tikzpicture}
\]
From this it follows that  $H_a$ is also symplectic, and it must thus be injective. A simple dimension count now finishes the argument.
\eproof

\noindent\emph{Proof of Theorem \ref{3-2SS}.}
We just need to prove (in the notation of subsection \ref{3-2PM}), that the natural map from
$\tilde{S}(X)$ to $\tilde{S}(X_e)$ induces a symplectic isomorphism between $\tilde{S}_r(X)$ and $\tilde{S}_r(X_e)$.
We do the symplectic reduction of $\tilde{S}(X)$ in two steps, the first one being only the reduction with respect to the extra edge $e$, and then the reduction with respect to all other edges. Thus, in the first step, we consider only $\tilde{S}(S)$ together with the natural map to $\tilde{S}(S_e)$. The gauge transformation corresponding to edge $e$ leaves invariant the dihedral angles on the edges of $S_e$, and it is a simple explicit check that the initial map descends to a symplectic isomorphism between  $\tilde{S}_r(S)$ and $\tilde{S}(S_e)$. Now, after accomplishing this identification, we remark that the symplectic spaces $\tilde{S}(X\setminus S)\times \tilde{S}_r(S)$ and $\tilde{S}(X_e)=\tilde{S}(X_e\setminus S_e)\times \tilde{S}(S_e)$ are trivially isomorphic, and thus so are their symplectic reductions over all remaining edges of $X$ which are in bijection with all edges of $X_e$.
\eproof

\section{The tetrahedral operator of quantum Teichm\"{u}ller theory}\label{qtt}
In this section, we recall the main ingredients of quantum Teichm\"{u}ller theory, following the approach of \cite{K1,K2,K3}. First, we consider the usual canonical quantization of $T^*(\bR^n)$ with the standard symplectic structure in the position representation, i.e. with respect to the vertical real polarization. The Hilbert space we get is of course just $L^2(\bR^n)$, but it will be convenient for us to consider instead the pre-Hilbert space $\S(\bR^n)$ and its dual space of temperate distributions $\S'(\bR^n)$. The position coordinates $q_i$ and momentum coordinates $p_i$ on
$T^*(\bR^n)$ upon quantization become operators $\POS_i$ and $\MOM_i$ acting on $\S(\bR^n)$ via the formulae
$$
\POS_i(f)(t) = t_if(t) \mbox{ and } \MOM_i(f)(t) = \frac{1}{2 \pi i}\frac{\partial}{\partial t_i} (f)(t),\quad \forall t\in \bR^n
$$
for all $f\in \S(\bR^n)$. It is known that these operators extend continuously to operators on $\S'(\bR^n)$, still satisfying the Heisenberg commutation relations
\begin{equation}
  \label{eq:hei-cr}
[\MOM_i,\MOM_j]=[\POS_i,\POS_j]=0,\quad  [\MOM_i,\POS_j]=(2\pi \imun)^{-1}\delta_{i,j}.
\end{equation}
For any $\alpha \in \bR$ we also define the weighted Schwartz space
$$
\S_\alpha(\bR^n) = e^{\alpha \rho} \S(\bR^n),
$$
where $\rho$ is a smooth function defined on $\bR^n$ such that $\rho$ coincides with the function $|x|$ on the complement of a compact subset of $\bR^n$.

Fix now $\la\in \bC$ such that $\re( \la)\ne0$.
By the usual spectral theorem, we can define operators
\[ \sfu_i=e^{2\pi\la \POS_i},\quad \sfv_i = e^{2\pi\la \MOM_i}
\]
which are contained in $\L(\S_\alpha(\bR^n), \S_{\alpha-\re(\la)}(\bR^n))$ for any $\alpha\in\bR$.
The corresponding commutation relations between $\sfu_i$ and $\sfv_j$ take the form
\[
[\sfu_i,\sfu_j]=[\sfv_i,\sfv_j]=0,\quad\sfu_i\sfv_j=e^{\imun 2\pi\la^2\delta_{i,j}}\sfv_j\sfu_i.
\]
Following \cite{K1}, we consider the operations for $\wv_i =(\sfu_i,\sfv_i)$, $i=1,2$,
\begin{equation}
  \label{eq:q-dot}
\wv_1\cdot\wv_2:=(\sfu_1\sfu_2,\sfu_1\sfv_2+\sfv_1)
\end{equation}
\begin{equation}
  \label{eq:q-star}
\wv_1*\wv_2:=(\sfv_1\sfu_2(\sfu_1\sfv_2+\sfv_1)^{-1},
\sfv_2(\sfu_1\sfv_2+\sfv_1)^{-1})
\end{equation}
\begin{proposition}[\cite{K1}]
Let $\psi(z)$ be some solution of the functional equation
\begin{equation}\label{functional-equation}
\psi(z+\IMUN\la/2)=\psi(z-\IMUN\la/2)(1+e^{2\pi\la z}),\quad z\in\mathbb{C}
\end{equation}
Then, the operator
\begin{equation}
  \label{eq:ptol}
\PTOLEMY=\PTOLEMY_{12}:=e^{2\pi\imun\MOM_1\POS_2}
\psi(\POS_1+\MOM_2-\POS_2)=
\psi(\POS_1-\MOM_1+\MOM_2)e^{2\pi\imun\MOM_1\POS_2}
\end{equation}
defines an element in $\L(\S(\bR^4),\S(\bR^4))$, which
satisfies the equations
\begin{equation}
  \label{eq:lin-ur}
\wv_1\cdot\wv_2\PTOLEMY=\PTOLEMY\wv_1,\quad
\wv_1*\wv_2\PTOLEMY=\PTOLEMY\wv_2
\end{equation}
\end{proposition}
\begin{proof}
That $\PTOLEMY\in \L(\S(\bR^4),\S(\bR^4))$ is seen by conjugating it with the Fourier transform in the $(p_1,p_2)$ directions. Then the operator becomes multiplication by a bounded function, which of course maps $\S(\bR^4)$ to $\S(\bR^4)$.

Equations~(\ref{eq:lin-ur}) follow from the following system of equations,
\begin{gather}
\PTOLEMY\POS_1=(\POS_1+\POS_2)\PTOLEMY\label{eq:def-ptol-1}\\
\PTOLEMY(\MOM_1+\MOM_2)=\MOM_2\PTOLEMY\label{eq:def-ptol-2}\\
\PTOLEMY(\MOM_1+\POS_2)=(\MOM_1+\POS_2)\PTOLEMY\label{eq:def-ptol-3}\\
\PTOLEMY e^{2\pi\la\MOM_1}=
(e^{2\pi\la(\POS_1+\MOM_2)}+e^{2\pi\la\MOM_1})\PTOLEMY\label{eq:def-ptol-4}
\end{gather}
Under substitution of (\ref{eq:ptol}), the first three equations become identities while the forth one reduces to the functional equation~(\ref{functional-equation}).
\end{proof}
One particular solution of (\ref{functional-equation}) is given by Faddeev's quantum dilogarithm \cite{F}
\begin{equation}
  \label{eq:qdl-subst}
  \psi(z)=\QDILOGI(z):=1/\QDILOG(z)
\end{equation}
The most important property of the operator~(\ref{eq:ptol}) with $\psi$ given by~(\ref{eq:qdl-subst}) is the pentagon identity
\begin{equation}\label{eq:pentagon}
\PTOLEMY_{12}\PTOLEMY_{13}\PTOLEMY_{23}=\PTOLEMY_{23}\PTOLEMY_{12}
\end{equation}
which follows from the five-term identity (\ref{eq:pent}) (see Appendix~A) satisfied  by $\QDILOG(z)$.
The indices in \eqref{eq:pentagon} have the standard meaning, for example, $\PTOLEMY_{13}$ is obtained from $\PTOLEMY_{12}$ by replacing $\MOM_2$ and $\POS_2$ by $\MOM_3$ and $\POS_3$ respectively,
and so on. In what follows, we always assume that the parameter $\la$ is chosen so that
\[
\hbar:=(\la+\la^{-1})^{-2}\in\bR_+
\]
\section{Charged tetrahedral operators}\label{CTO}
For any positive real $a$ and $c$ such that $b:=\frac12-a-c$ is also positive, we define the charged $T$-operators
\begin{equation}\label{eq:charged-T}
\PTOLEMY(a,c)=e^{-\pi\imun\cla^2(4(a-c)+1)/6}e^{4\pi\imun\cla(c\POS_2-a\POS_1)}\PTOLEMY
e^{-4\pi\imun\cla(a\MOM_2+c\POS_2)}
\end{equation}
and
\[
\bar\PTOLEMY(a,c)=e^{\pi\imun\cla^2(4(a-c)+1)/6}e^{-4\pi\imun\cla(a\MOM_2+c\POS_2)}
\bar\PTOLEMY e^{4\pi\imun\cla(c\POS_2-a\POS_1)}
\]
where $\bar\PTOLEMY:=\PTOLEMY^{-1}$ and
$$
\cla := \imun (\la + \la^{-1})/2.
$$
These are direct analogues of the charged $6j$-symbols of \cite{K4}, see also \cite{GKT} for a general theory of charged $6j$-symbols. It is elementary to prove that $T(a,c) : \S(\bR^2) \ra \S(\bR^2)$ and that $\bar\PTOLEMY(a,c) : \S(\bR^2) \ra \S(\bR^2)$.

Substituting \eqref{eq:ptol}, we obtain
\[
\PTOLEMY(a,c)=e^{2\pi\imun\MOM_1\POS_2}\cff_{a,c}(\POS_1-\POS_2+\MOM_2)
\]
where
\[
\cff_{a,c}(x):=\psi(x-2\cla(a+c))e^{-4\pi\imun\cla a(x-\cla(a+c))}e^{-\pi\imun\cla^2(4(a-c)+1)/6}
\]
We have the following formula\footnote{From now on, we freely switch to Dirac's bra-ket notation which is convenient in calculations.} for $\PTOLEMY(a,c)\in \S'(\bR^4)$
\[
\langle x_0,x_2\vert\PTOLEMY(a,c)\vert x_1,x_3\rangle=\delta(x_0+x_2-x_1)\tilde\cff'_{a,c}(x_3-x_2)e^{2\pi\imun x_0(x_3-x_2)}
\]
where
\[
\tilde\cff'_{a,c}(x):=e^{-\pi\imun x^2}\tilde\cff_{a,c}(x),\quad \tilde\cff_{a,c}(x):=\int_\REALS\cff_{a,c}(y)e^{-2\pi\imun xy}dy
\]
Notice that the conditions we imposed on $a$ and $c$ ensure that the Fourier integral here is absolutely convergent. The Fourier transformation formula for the Faddeev's quantum dilogarithm (see Appendix~A) leads to the identity
\[
\tilde\cff'_{a,c}(x)=e^{-\frac{\pi\imun}{12}}\cff_{c,b}(x),
\]
recalling that $b:=\frac12-a-c$. Moreover, with respect to complex conjugation, we also have
\[
\overline{\cff_{a,c}(x)}=e^{-\frac{\pi\imun}{6}}e^{\pi\imun x^2}
\cff_{c,a}(-x)=e^{-\frac{\pi\imun}{12}}\tilde\cff_{b,c}(-x),
\]
These can be combined to calculate that
\[
\overline{\tilde\cff'_{a,c}(x)}=e^{\frac{\pi\imun}{12}}\overline{\cff_{c,b}(x)}
=e^{-\frac{\pi\imun}{12}}e^{\pi\imun x^2}\cff_{b,c}(-x).
\]
We can use this to obtain the following formula of $\bar\PTOLEMY(a,c)$:
\begin{multline}\label{eq:tbarkernel}
\langle x,y\vert\bar\PTOLEMY(a,c)\vert u,v\rangle=
\overline{\langle u,v\vert\PTOLEMY(a,c)\vert x,y\rangle}\\
=\delta(u+v-x)\overline{\tilde\cff'_{a,c}(y-v)}e^{-2\pi\imun u(y-v)}\\
=\delta(u+v-x)\cff_{b,c}(v-y)e^{-\frac{\pi\imun}{12}}e^{\pi\imun (v-y)^2}
e^{-2\pi\imun u(y-v)}
\end{multline}

\section{Charged Pentagon identity}\label{CPI}
\begin{proposition}\label{charged-pentagon}
The following charged pentagon equation is satisfied
\begin{equation}\label{eq:charged-pentagon}
  \PTOLEMY_{12}(a_4,c_4) \, \PTOLEMY_{13}(a_2,c_2) \PTOLEMY_{23}(a_0,c_0) =
e^{\pi\imun\cla^2P_e/3}\PTOLEMY_{23}(a_1,c_1) \PTOLEMY_{12}(a_3,c_3)
\end{equation}
where
\[
P_e=2(c_0+a_2+c_4)-\frac12
\]
and $a_0,a_1, a_2,a_3, a_4, c_0, c_1, c_2, c_3, c_4
 \in\REALS$  are such that
\begin{equation}\label{pentagonconditions}
 a_1=a_0+a_2,\ a_3=a_2+a_4,\ c_1=c_0+a_4,\ c_3=a_0+c_4,\ c_2= c_1+c_3.
\end{equation}
\end{proposition}
This is direct analogue of the charged pentagon identity of \cite{K4}, see also \cite{GKT}.
\begin{proof}
We have that
\[
\PTOLEMY(a,c)=\nu(a-c)\PTOLEMY'(a,c)
\]
where
\[
\nu(x):=e^{-\pi\imun\cla^2(4x+1)/6}
\]
and
\[
\PTOLEMY'(a,c):=\xi^{a\POS_1-c\POS_2}\PTOLEMY
\xi^{a\MOM_2+c\POS_2},
\]
where $\xi:=e^{-4\pi\imun\cla}$.
It is straightforward to check that under conditions~\eqref{pentagonconditions} we have the identity
\[
\frac{\nu(a_4-c_4)\nu(a_2-c_2)\nu(a_0-c_0)}{\nu(a_1-c_1)\nu(a_3-c_3)}=e^{\pi\imun\cla^2P_e/3}
\]
so that equation~\eqref{eq:charged-pentagon} is equivalent to the equation
\begin{equation}\label{eq:charged-pentagon-prime}
\PTOLEMY_{12}'(a_4,c_4) \, \PTOLEMY_{13}'(a_2,c_2) \PTOLEMY_{23}'(a_0,c_0) =
\PTOLEMY_{23}'(a_1,c_1) \PTOLEMY_{12}'(a_3,c_3).
\end{equation}
For the right hand side of the latter, we have
\begin{multline}\label{eq:cal-rhs}
\PTOLEMY_{23}'(a_1,c_1) \PTOLEMY_{12}'(a_3,c_3)=
\xi^{a_1\POS_2-c_1\POS_3}\PTOLEMY_{23}\xi^{a_1\MOM_3+c_1\POS_3}
\xi^{a_3\POS_1-c_3\POS_2}\PTOLEMY_{12}\xi^{a_3\MOM_2+c_3\POS_2}\\
=
\xi^{a_3\POS_1+a_1\POS_2-c_1\POS_3}\underline{\PTOLEMY_{23}
\xi^{-c_3\POS_2}}\PTOLEMY_{12}\xi^{a_3\MOM_2+c_3\POS_2+a_1\MOM_3+c_1\POS_3}\\
=
\xi^{a_3\POS_1+(a_1-c_3)\POS_2-(c_1+c_3)\POS_3}\PTOLEMY_{23}
\PTOLEMY_{12}\xi^{a_3\MOM_2+c_3\POS_2+a_1\MOM_3+c_1\POS_3}
\end{multline}
where the underlined fragment is transformed by using equation~\eqref{eq:def-ptol-1}.
For the left hand side of equation~\eqref{eq:charged-pentagon-prime}, by underlining the parts to be transformed either by trivial commutativity or the Heisenberg commutation relations or else according to one of equations~\eqref{eq:def-ptol-1}--\eqref{eq:def-ptol-3}, we have that
\begin{multline}\label{eq:cal-lhs}
 \xi^{c_4\POS_2-a_4\POS_1}\PTOLEMY_{12}'(a_4,c_4) \, \PTOLEMY_{13}'(a_2,c_2) \PTOLEMY_{23}'(a_0,c_0)\xi^{-a_0\MOM_3-c_0\POS_3}\\=
\underline{\PTOLEMY_{12}
\xi^{a_4\MOM_2+c_4\POS_2}
\xi^{a_2\POS_1-c_2\POS_3}
\PTOLEMY_{13}}
\xi^{a_2\MOM_3+c_2\POS_3}
\xi^{a_0\POS_2-c_0\POS_3}\PTOLEMY_{23}
\\
=
\xi^{-c_2\POS_3}\underline{\PTOLEMY_{12}\xi^{a_2\POS_1}}
\PTOLEMY_{13}\underline{\xi^{a_4\MOM_2+c_4\POS_2}
\xi^{a_2\MOM_3+c_2\POS_3}
\xi^{a_0\POS_2-c_0\POS_3}}\PTOLEMY_{23}\\
=
\xi^{a_2(\POS_1+\POS_2)-c_2\POS_3}\PTOLEMY_{12}
\PTOLEMY_{13}\xi^{\cla (a_2c_2+a_4c_4)+a_2\MOM_3}\xi^{a_4(\MOM_2+\POS_3)}
\underline{\xi^{c_3(\POS_2+\POS_3)}\PTOLEMY_{23}}\\
=
\xi^{a_2(\POS_1+\POS_2)-c_2\POS_3}\PTOLEMY_{12}
\PTOLEMY_{13}\xi^{\cla (a_2c_2+a_4c_4)+a_2\MOM_3}\underline{\xi^{a_4(\MOM_2+\POS_3)}
\PTOLEMY_{23}}
\xi^{c_3\POS_2}\\
=
\xi^{a_2(\POS_1+\POS_2)-c_2\POS_3}\PTOLEMY_{12}
\PTOLEMY_{13}\underline{\xi^{\cla (a_2c_2+a_4c_4)+a_2\MOM_3}
\PTOLEMY_{23}}\xi^{a_4(\MOM_2+\POS_3)}
\xi^{c_3\POS_2}\\
=
\xi^{a_2(\POS_1+\POS_2)-c_2\POS_3}\PTOLEMY_{12}
\PTOLEMY_{13}
\PTOLEMY_{23}\xi^{\cla (a_2c_2+a_4c_4)+a_2(\MOM_2+\MOM_3)}\xi^{a_4(\MOM_2+\POS_3)}
\xi^{c_3\POS_2}.
\end{multline}
By comparing the equations~\eqref{eq:cal-rhs} and \eqref{eq:cal-lhs} and using the pentagon identity~\eqref{eq:pentagon}, we conclude that equation~\eqref{eq:charged-pentagon-prime} is equivalent to the identity
\[
\xi^{\cla (a_2c_2+a_4c_4)+a_2(\MOM_2+\MOM_3)}\xi^{a_4(\MOM_2+\POS_3)}
\xi^{c_3\POS_2+a_0\MOM_3+c_0\POS_3}=\xi^{a_3\MOM_2+c_3\POS_2+a_1\MOM_3+c_1\POS_3},
\]
which, in turn, is equivalent to the quadratic scalar identity
\[
a_2c_2+a_4c_4+a_4(a_0-c_3)=a_2(a_4+c_0+c_3).
\]
The latter is verified straightforwardly by using equations~\eqref{pentagonconditions}.
\end{proof}

\section{Fundamental Lemma}

Define two complex temperate ket-distributions $\sA\equiv\vert\sA\rangle,\sB\equiv\vert\sB\rangle\in\S'(\bR^2)$ by the formulae
\[
\langle x,y\vert\sA\rangle=\delta(x+y)e^{\imun\pi(x^2+\frac1{12})},\quad
 \langle x,y\vert\sB\rangle=e^{\imun\pi(x-y)^2}
\]
We also define two bra-distributions $\bar\sA\equiv\langle\bar\sA\vert,\bar\sB\equiv\langle\sB\vert$ by the formulae
\[
\langle\bar\sA\vert x,y\rangle= \overline{\langle x,y\vert\sA\rangle},\quad
\langle\bar\sB\vert x,y\rangle= \overline{\langle x,y\vert\sB\rangle}
\]
\begin{lemma}[Fundamental Lemma]\label{fund-lem} The following identities\footnote{We now also switch to the physicist integral notation $\int dx f(x)$ and we further also use it to denote push forward of distributions.} are satisfied:
\begin{equation}\label{eq:p01}
\int_{\bR^2}dsdt\,\langle\bar\sA\vert v,s\rangle\langle x,s\vert\PTOLEMY(a,c)\vert u,t\rangle
\langle t,y\vert\sA\rangle=\langle x,y\vert\bar\PTOLEMY(a,b)\vert u,v\rangle
\end{equation}
\begin{equation}\label{eq:p12}
\int_{\bR^2}dsdt\,\langle\bar\sA\vert u,s\rangle\langle s,x\vert\PTOLEMY(a,c)\vert v,t\rangle
\langle t,y\vert\sB\rangle=\langle x,y\vert\bar\PTOLEMY(b,c)\vert u,v\rangle
\end{equation}
\begin{equation}\label{eq:p23}
\int_{\bR^2}dsdt\,\langle\bar\sB\vert u,s\rangle\langle s,y\vert\PTOLEMY(a,c)\vert t,v\rangle
\langle t,x\vert\sB\rangle=\langle x,y\vert\bar\PTOLEMY(a,b)\vert u,v\rangle
\end{equation}
where $a,b,c\in\bR_+$ are such that $a+b+c=\frac12$.
\end{lemma}
These identities are direct analogues of the identities (3.8)--(3.10) of \cite{K4} in the root of unity case, see also the fundamental lemma~6 of \cite{GKT}.
\begin{proof} First of all, we notice that the Wave front set condition from Theorem \ref{mult} is satisfied in these three cases, so we can multiply these distributions. Further we notice that the products extend appropriately, so that the push forwards on the left hand sides of (\ref{eq:p01}) -- (\ref{eq:p23}) can be performed.

For the left hand side of equation~\eqref{eq:p01}, we have that
\begin{multline*}
\int_{\bR^2}dsdt\,\langle\bar\sA\vert v,s\rangle\langle x,s\vert\PTOLEMY(a,c)\vert u,t\rangle
\langle t,y\vert\sA\rangle=\langle x,-v\vert\PTOLEMY(a,c)\vert u,-y\rangle
e^{\imun\pi (y^2-v^2)}\\
=\delta(x-v-u)\tilde\cff'_{a,c}(-y+v)e^{\imun 2\pi x(-y+v)}e^{\imun\pi (y^2-v^2)}\\
=\delta(u+v-x)\cff_{c,b}(v-y)e^{\imun\pi(v-y)(2u+v-y)}e^{-\frac{\imun\pi}{12}},
\end{multline*}
where the last expression is equal to the right hand side of equation~\eqref{eq:p01} due to
equation~\eqref{eq:tbarkernel}.

For the left hand side of equation~\eqref{eq:p12}, we have that
\begin{multline*}
\int_{\bR^2}dsdt\,\langle\bar\sA\vert u,s\rangle\langle s,x\vert\PTOLEMY(a,c)\vert v,t\rangle
\langle t,y\vert\sB\rangle\\
=e^{-\frac{\imun\pi}{12}-\imun\pi u^2}\int_{\bR}dt\,\langle -u,x\vert\PTOLEMY(a,c)\vert v,t\rangle
\langle t,y\vert\sB\rangle\\
=e^{-\frac{\imun\pi}{12}-\imun\pi u^2}\delta(-u+x-v)\int_{\bR}dt\,
\tilde\cff'_{a,c}(t-x)e^{-\imun 2\pi u(t-x)}e^{\imun\pi(t-y)^2}\\
=e^{-\frac{\imun\pi}{12}-\imun\pi u^2}\delta(u+v-x)\int_{\bR}dt\,
\tilde\cff_{a,c}(t)e^{-\imun\pi t^2}e^{-\imun 2\pi ut}e^{\imun\pi(t+x-y)^2}\\
=e^{-\frac{\imun\pi}{12}}\delta(u+v-x)e^{\imun\pi((x-y)^2-u^2)}\int_{\bR}dt\,
\tilde\cff_{a,c}(t)e^{\imun2\pi t(v-y)}\\
=e^{-\frac{\imun\pi}{12}}\delta(u+v-x)e^{\imun\pi(x-y-u)(x-y+u)}
\cff_{a,c}(v-y)\\
=e^{-\frac{\imun\pi}{12}}\delta(u+v-x)e^{\imun\pi(v-y)(2u+v-y)}
\cff_{a,c}(v-y),
\end{multline*}
where the last expression is equal to the right hand side of \eqref{eq:p12} due to
\eqref{eq:tbarkernel}.

Finally, for the left hand side of equation~\eqref{eq:p23}, we have that
\begin{multline*}
\int_{\bR^2}dsdt\,\langle\bar\sB\vert u,s\rangle\langle s,y\vert\PTOLEMY(a,c)\vert t,v\rangle
\langle t,x\vert\sB\rangle\\
=\tilde\cff'_{a,c}(v-y)\int_{\bR}ds\,\langle\bar\sB\vert u,s\rangle
e^{\imun 2\pi s(v-y)}
\langle s+y,x\vert\sB\rangle\\
=\tilde\cff'_{a,c}(v-y)e^{\imun\pi(y-x-u)(y-x+u)}\int_{\bR}ds\,
e^{\imun 2\pi s(u+v-x)}\\
=e^{-\frac{\imun\pi}{12}}\cff_{c,b}(v-y)e^{\imun\pi(y-v-2u)(y-v)}
\delta(u+v-x),
\end{multline*}
where the last expression is equal to the right hand side of  \eqref{eq:p12} due to
\eqref{eq:tbarkernel}.
\end{proof}

\section{TQFT rules and symmetries of tetrahedral partitions functions}\label{tqft-rules}
We consider oriented surfaces with a restricted (oriented) cellular structure: all 2-cells are either bigons or triangles with their natural cellular structures. For example, the unit disk $D$ in $\mathbb{C}$ has four different bigon structures with two 0-cells $\{e^0_\pm(*)=\pm1\}$, two 1-cells $e^1_\pm\colon [0,1]\to D$ given by
\[
e^1_+(t)=e^{i\pi t}\ \mathrm{or}\ e^1_+(t)=-e^{-i\pi t},\quad
e^1_-(t)=e^{-i\pi t}\ \mathrm{or}\ e^1_-(t)=-e^{i\pi t}.
\]
An \emph{inessential bigon} is the one
where both edges are parallel. For the unit disk the two cell structures with $\{e^1_+(t)=e^{i\pi t},e^1_-(t)=e^{-i\pi t}\}$ and $\{e^1_+(t)=-e^{-i\pi t},e^1_-(t)=-e^{i\pi t}\}$ are inessential. Inessential bigons can be eliminated by naturally contracting them to a 1-cell, which in the case of the unit disk corresponds to projection to the real axis, $z\mapsto \re z$.

For triangles we forbid cyclic orientation of the 1-cells. For example, the unit disk admits eight different triangle structures with 0-cells at third roots of unity $\{e^0_0(*)=1, e^0_\pm(*)=e^{\pm 2\pi i/3}\}$, but only six of them are admissible.

We start defining our TQFT by associating the vector space $\mathbb{C}$ to all bigons, and the spaces $\S'(\REALS)$ to all triangles. For any oriented surface $\Sigma$, connected or not, with boundary or without, with a fixed cellular structure where all 2-cells are either essential bigons or admissible triangles, we associate the space $\S'(\bR^{\Delta_2(\Sigma)})$ where in $\Delta_2(\Sigma)$ we include only triangular cells.

 Let $T$ be a shaped tetrahedron in $\REALS^3$ with ordered vertices $v_i$, enumerated by the integers $0,1,2,3$.
We define the partition function $\Psi(T):=\widetilde{\parfun(T)}$ by the following formula
\begin{equation}\label{tet-fun}
\langle x\vert\Psi(T)\rangle=
\left\{
\begin{array}{cc}
\langle x_0,x_2\vert\PTOLEMY(c(v_0v_1),c(v_0v_3))\vert x_1,x_3\rangle&\mathrm{if}\ \sign(T)=1;\\
\langle x_1,x_3\vert\bar\PTOLEMY(c(v_0v_1),c(v_0v_3))\vert x_0,x_2\rangle&\mathrm{if}\ \sign(T)=-1.
\end{array}
\right.
\end{equation}
where
\[
x_i=x(\partial_iT),\quad i\in\{0,1,2,3\},
\]
and
\[
 c:=\frac1{2\pi}\alpha_T\colon \Delta_1(T)\to \REALS_+.
 \]
Since this definition is adaptation of the one of \cite{K4} to the infinite dimensional setting of quantum Teichm\"uller theory, the map $c$ will be called \emph{charge} which is actually nothing else but the shape structure with the dihedral angles measured by circumference of a circle of radius $\frac1{2\pi}$.

Next, we consider cones over essential bigons with induced cellular structure. Overall, the condition of admissibility of triangular faces leaves four isotopy classes of such cones. Let us describe them by using the cone in $\mathbb{R}^3\simeq\mathbb{C}\times\mathbb{R}$ over the unit disk in $\mathbb{C}$ with the apex at $(0,1)\in\mathbb{C}\times\mathbb{R}$. The four possible cellular structures are identified by three 0-cells $\{e^0_\pm(*)=(\pm1,0), e^0_0(*)=(0,1)\}$, and four 1-cells
\[
\{e^1_{0\pm}(t)=(\pm e^{i\pi t},0),\ e^1_{1\pm}(t)=(\pm(1-t),t)\}
\]
or
\[
\{e^1_{0\pm}(t)=(\mp e^{-i\pi t},0),\ e^1_{1\pm}(t)=(\pm(1-t),t)\}
\]
or
\[
\{e^1_{0\pm}(t)=(\pm e^{i\pi t},0),\ e^1_{1\pm}(t)=(\pm t,1-t)\}
\]
or else
\[
\{e^1_{0\pm}(t)=(\mp e^{-i\pi t},0),\ e^1_{1\pm}(t)=(\pm t,1-t)\}.
\]
Let us call them cones of type $A_+$, $A_-$, $B_+$, and $B_-$, respectively.
We add TQFT rules by associating to these cones the following complex temperate distributions over $\bR^2$:
\[
\langle x,y\vert\Psi(A_\pm)\rangle=\delta(x+y)e^{\pm\pi\imun(x^2+\frac1{12})},\quad
\langle x,y\vert\Psi(B_\pm)\rangle=e^{\pm\pi\imun(x-y)^2}
\]
Notice that our cones are symmetric with respect to rotation by the angle $\pi$ around the vertical coordinate axis, and this symmetry corresponds to the symmetry with respect to exchange of the arguments $x$ and $y$.

Now we give a TQFT description of the tetrahedral symmetries of the tetrahedral partition functions generated by the identities~\eqref{eq:p01}--\eqref{eq:p23}. Let us take a positive tetrahedron with vertices $v_0,\ldots,v_3$. The edge $v_0v_1$ is incident to two faces opposite to $v_2$ and $v_3$. We can glue two cones, one of type $A_+$ and another one of type $A_-$, to these two faces in the way that one of the edges of the base bigons are glued to the initial edge $v_0v_1$, of course, by respecting all the orientations. Namely, we glue a cone of type $A_+$ to the face opposite to vertex $v_3$ so that the apex of the cone is glued to vertex $v_2$, and we glue a cone of type $A_-$ to the face opposite to vertex $v_2$ so that the apex of the cone is glued to vertex $v_3$. Finally, we can glue naturally the two bigons with each other by pushing continuously the initial edge inside the ball and eventually closing the gap like a book. The result of all these operations is that we obtain a negative tetrahedron, where the only difference with respect to the initial tetrahedron is that the orientation of the initial edge $v_0v_1$ has changed and this corresponds to changing the order of these vertices. Notice that as these vertices are neighbors, their exchange does not affect the orientation of all other edges. Identity~\eqref{eq:p01} corresponds exactly to these geometric operations  on the level of the tetrahedral partition functions.

 Similarly, we can describe the transformations with respect to change of orientations of the edges $v_1v_2$ and $v_2v_3$ and relate them to identities~\eqref{eq:p12} and \eqref{eq:p23}, respectively. Altogether, these three transformations correspond to canonical generators of the permutation group of four elements which is the complete symmetry group of a tetrahedron.

\section{Gauge transformatoin properties of the partition function}\label{GTP}

Let $SP_n$ be the suspension of an $n$-gone with its natural triangulation into $n$ tetrahedra having one common central interior edge $e$. For each tetrahedron in $SP_n$, let us choose the vertex order so that the tetrahedron is positive and the common edge $e$ connects its last two vertices. With this choice, and enumerating the tetrahedra cyclically, the quantum partition function of $SP_n$ is written as follows
\[
\parfun(SP_n,\vec{a},\vec{c})=\tr_0(\PTOLEMY_{01}(a_1,c_1)\PTOLEMY_{02}(a_2,c_2)\cdots \PTOLEMY_{0n}(a_n,c_n)),
\]
where $\vec{a}=(a_1,\ldots,a_n)$ and $\vec{c}=(c_1,\ldots,c_n)$.
The total charge $Q_e$ around the interior edge $e$ (the weight divided by $2\pi$) is given by the formula
\[
Q_e=a_1+a_2+\cdots+a_n,
\]
so that the gauge transformation corresponding to edge $e$ shifts simultaneously all $c_i$'s by one and the same amount:
\[
\parfun(SP_n,\vec{a},\vec{c})\to \parfun(SP_n,\vec{a},\vec{c}+\lambda\vec{1}),
\]
where $\vec{1}=(1,\ldots,1)$.
\begin{proposition}\label{gauge-trans}
One has the following equality
\[
\parfun(SP_n,\vec{a},\vec{c}+\lambda\vec{1})=\parfun(SP_n,\vec{a},\vec{c})e^{2\pi\imun \cla^2(n-6Q_e)\lambda/3}
\]
\end{proposition}
This easily follows from the following lemma and the cyclic property of the trace.
\begin{lemma}
\[
\PTOLEMY(a,c+\lambda)=e^{-4\pi\imun\cla\lambda\MOM_1}\PTOLEMY(a,c)
e^{4\pi\imun\cla\lambda\MOM_1}e^{2\pi\imun\cla^2(1-6a)\lambda/3}
\]
\end{lemma}
\begin{proof}
By using formula~\eqref{eq:charged-T}, we have
\begin{multline*}
e^{2\pi\imun\cla^2(6a-1)\lambda/3}\PTOLEMY(a,c+\lambda)\\
=e^{2\pi\imun\cla^2(6a-1)\lambda/3}
e^{-\pi\imun\cla^2(4(a-c-\lambda)+1)/6}
e^{4\pi\imun\cla((c+\lambda)\POS_2-a\POS_1)}\PTOLEMY
e^{-4\pi\imun\cla(a\MOM_2+(c+\lambda)\POS_2)}\\
=e^{4\pi\imun\cla\lambda\POS_2}e^{-\pi\imun\cla^2(4(a-c)+1)/6}
e^{4\pi\imun\cla(c\POS_2-a\POS_1)}\PTOLEMY
e^{-4\pi\imun\cla(a\MOM_2+c\POS_2)}e^{-4\pi\imun\cla\lambda\POS_2}\\
=
e^{4\pi\imun\cla\lambda\POS_2}\PTOLEMY(a,c)e^{-4\pi\imun\cla\lambda\POS_2}=
e^{-4\pi\imun\cla\lambda\MOM_1}\PTOLEMY(a,c)e^{4\pi\imun\cla\lambda\MOM_1}
\end{multline*}
\end{proof}

\section{Geometric constraints on partition functions}\label{GC}
\subsection{Operator-vector correspondence}
Let $I$ and $J$ be two sets and
$$
\mathsf{A}\in\L\left(\S(\REALS^I),\S'(\REALS^J)\right).
$$
Under isomorphism~\eqref{iso}, we have that
\[
\widetilde{\mathsf{A}\POS_i}=\POS_{i}\widetilde{\mathsf{A}},\quad \widetilde{\mathsf{A}\MOM_i}=-\MOM_{i}\widetilde{\mathsf{A}},\quad \forall i\in I.
\]
Indeed,
\[
\langle x\sqcup y\vert\widetilde{\mathsf{A}\POS_i}\rangle=\langle x\vert\mathsf{A}\POS_i\vert y\rangle=y(i)\langle x\vert\mathsf{A}\vert y\rangle=y(i)\langle x\sqcup y\vert\widetilde{\mathsf{A}}\rangle=\langle x\sqcup y\vert\POS_i\widetilde{\mathsf{A}}\rangle,
\]
and
\begin{multline*}
\langle x\sqcup y\vert\widetilde{\mathsf{A}\MOM_i}\rangle=\langle x\vert\mathsf{A}\MOM_i\vert y\rangle=-\frac1{2\pi\imun}\frac{\partial}{\partial y(i)}\langle x\vert\mathsf{A}\vert y\rangle\\
=-\frac1{2\pi\imun}\frac{\partial}{\partial y(i)}\langle x\sqcup y\vert\widetilde{\mathsf{A}}\rangle
=-\langle x\sqcup y\vert\MOM_i\widetilde{\mathsf{A}}\rangle,
\end{multline*}
Similarly to the uncharged $T$-operator $\PTOLEMY$, the charged $T$-operator $\PTOLEMY(a,c)$ satisfies the identities \eqref{eq:def-ptol-1}--\eqref{eq:def-ptol-3}, i.e.
\begin{gather*}
\PTOLEMY(a,c)\POS_1=(\POS_1+\POS_2)\PTOLEMY(a,c),\\
\PTOLEMY(a,c)(\MOM_1+\MOM_2)=\MOM_2\PTOLEMY(a,c),\\
\PTOLEMY(a,c)(\MOM_1+\POS_2)=(\MOM_1+\POS_2)\PTOLEMY(a,c).
\end{gather*}
Through the above identification, the corresponding temperate distribution satisfies the identities
\begin{equation}\label{eq:h1}
(\POS_1+\POS_2-\POS_3)\vert\widetilde{\PTOLEMY(a,c)}\rangle=0,
\end{equation}
\begin{equation}\label{eq:h2}
(\MOM_2+\MOM_3+\MOM_4)\vert\widetilde{\PTOLEMY(a,c)}\rangle=0,
\end{equation}
\begin{equation}\label{eq:h3}
(\MOM_1+\POS_2+\MOM_3-\POS_4)\vert\widetilde{\PTOLEMY(a,c)}\rangle=0.
\end{equation}

\subsection{An operator-valued cohomology class}
Identities~\eqref{eq:h1}--\eqref{eq:h3} give rise to the following geometric properties of partition functions.

Given an oriented triangulated pseudo 3-manifold $X$. Let $\gamma$ be an oriented normal segment (in the sense of normal surface theory) in triangle $t\in\Delta_2(\partial X)$ parallel to oriented side $\partial_j t$. We associate to $\gamma$ an operator  $\widehat{\gamma}\in\End\left(\S'(\REALS^{\Delta_2(\partial X)})\right)$ defined by the formula
\[
\widehat{\gamma}:=\left\{
\begin{array}{cl}
-(-1)^{\sign(t)}\MOM_t&\mathrm{if}\ j=0;\\
\POS_t-(-1)^{\sign(t)}\MOM_t&\mathrm{if}\ j=1;\\
\POS_t&\mathrm{if}\ j=2.
\end{array}\right.
\]
We also define
\[
\widehat{-\gamma}=-\widehat{\gamma},
\]
where $-\gamma$ is $\gamma$ taken with opposite orientation, and
 \[
 \widehat{\gamma_1\sqcup\gamma_2}=\widehat{\gamma_1}+\widehat{\gamma_2}.
 \]
This correspondence permits to define a unique element
\[
\myhom(X)\in H^1\left(\partial X\setminus\Delta_0(\partial X),\End\left(\S'(\REALS^{\Delta_2(\partial X)})\right)\right)
\]
(here $\End\left(\S'(\REALS^{\Delta_2(\partial X)})\right)$ is considered as an additive Abelian group) defined by the formula
\[
\langle\myhom(X),\gamma\rangle=\sum_{T\in\Delta_2(\partial X)} \widehat{\gamma\cap T}
\]
where $\gamma$ is any oriented normal curve in $\partial X\setminus\Delta_0(\partial X)$ representing a class in $H_1\left(\partial X\setminus\Delta_0(\partial X),\INTEGERS\right)$.
\begin{proposition}
The class $\myhom(X)$ is such that
\[
[\langle\myhom(X),\gamma_1\rangle,\langle\myhom(X),\gamma_2\rangle]=
\frac1{\pi\imun}\gamma_1\cdot\gamma_2,
\]
where $\gamma_1\cdot\gamma_2$ is the algebraic intersection index of $\gamma_1$ and $\gamma_2$.
\end{proposition}
\begin{proof}
As in the proof of Theorem \ref{Shape}, we deform the curves $\gamma_1$ and $\gamma_2$, so that each intersection point becomes the midpoint of an edge shared by two triangles. Then both curves get contributions from these triangles to the commutation relation between the associated operators, and we get the stated formula.
\end{proof}

Now, relations~\eqref{eq:h1}--\eqref{eq:h3} are interpreted and generalized as the following geometric constraints on partition functions.

\begin{proposition}
Let $X$ be a shaped 3-manifold for which the partition function $\parfun(X)$ is a well defined distribution. Let $L_X$ be the kernel of the group homomorphism
\[
i_*\colon H_1\left(\partial X\setminus\Delta_0(\partial X),\INTEGERS\right)\to
H_1(X\setminus\Delta_0(X),\INTEGERS)
\]
induced by the inclusion map
\[
i\colon \partial X\setminus\Delta_0(\partial X)\hookrightarrow X\setminus\Delta_0(\partial X).
\]
Then,
\begin{equation}\label{eq:ann}
\langle\myhom(X),\gamma\rangle\vert \widetilde{\parfun(X)}\rangle=0,\quad \forall \gamma\in L_X.
\end{equation}
\end{proposition}
\begin{proof}
It is a straightforward to check that the statement is true for the tetrahedral partition functions $\PTOLEMY(a,c)$ and $\bar\PTOLEMY(a,c)$.
Assume that $Y$ satisfies the statement and that  $X$ is obtained from $Y$ by an \emph{elementary gluing} corresponding to identification of two triangles of opposite signs $t_1,t_2\in\Delta_2(\partial Y)$.
Denoting
\[
V(*):=\S'(\REALS^{\Delta_2(*)}),
\]
we have a linear map
\[
\cont_{t_1,t_2}\colon V(\partial Y)\to V(\partial X)
\]
defined by the formula
\[
\langle x\vert\cont_{t_1,t_2}\vert f\rangle=\int_{\REALS}ds\, \langle y_{x,s}\vert f\rangle
\]
where
\[
y_{x,s}(t)=\left\{
\begin{array}{cl}
s,&\mathrm{if}\ t=t_i,\ i\in\{1,2\};\\
x(t),&\mathrm{otherwise}.
\end{array}\right.
\]
With this map, we have that
\[
\vert\Psi(X)\rangle=\cont_{t_1,t_2}\vert\Psi(Y)\rangle,\quad \Psi:=\widetilde{\parfun}.
\]
If $\gamma_1$ and $\gamma_2$ are oriented normal arcs in $t_1$ and $t_2$ respectively which are identified upon gluing, then it is easily verified that
\[
\cont_{t_1,t_2}\widehat{\gamma_1}=\cont_{t_1,t_2}\widehat{\gamma_2}.
\]
Now, let $D$ be a normal surface in $X$ with $\partial D\subset\partial X$. The pre-image of $D$ in $Y$ is a finite system of normal surfaces $D_1,\ldots,D_n$, and the corresponding system of normal curves $\gamma:=\cup_{i=1}^n\partial D_i$ in $\partial Y$ is such that two sets of oriented normal arcs in the triangles $t_1$ and $t_2$,
\[
\bigcup_{j=1}^m\gamma_{i,j}=\gamma\cap t_i,\quad i\in\{1,2\},
\]
are in bijection induced by the identification of $t_1$ and $t_2$ so that
\[
\cont_{t_1,t_2}\widehat{\gamma_{1,j}}=\cont_{t_1,t_2}\widehat{\gamma_{2,j}},\quad j\in\{1,\ldots,m\}.
\]
Thus, we have the following operator identity
\[
\langle\myhom(X),\partial D\rangle\cont_{t_1,t_2}=\sum_{i=1}^n\cont_{t_1,t_2}\langle\myhom(Y),\partial D_i\rangle
\]
which permits to conclude that $\langle\myhom(X),\partial D\rangle\vert\Psi(X)\rangle=0$ as soon as $Y$ verifies equations~\eqref{eq:ann}.

To finish the proof, we remark that for any shaped 3-manifold $X$, there exists a finite sequence of shaped 3-manifolds
\[
X'=X_0,X_1,\ldots,X_n=X,
\]
where for any $i\in\{1,\ldots,n\}$, $X_i$ is obtained from $X_{i-1}$ by an elementary gluing.
\end{proof}

\section{Convergence of partition functions}\label{Conv}
\begin{theorem}\label{convergence}
For any  admissible  shaped 3-manifold $X$, the partition function $\parfun(X)$ is a well defined temperate distribution. In particular, if $\partial X=\varnothing$, then  $\parfun(X)\in\COMPLEXS$.
\end{theorem}
\begin{proof}Let $X$ be a shaped 3-manifold such that $\Psi(X)$ is the distribution as above, which satisfies equations~\eqref{eq:ann}. Let $\{\gamma_i\}_{i\in I}$ be a system of disjoint simple closed curves in $\partial X$ corresponding to a linear basis of $L_{X}$. We denote
\[
\mathsf{a}_i:=\langle\myhom(X),\gamma_i\rangle,\quad i\in I,
\]
and we choose a maximal set of operators $\mathsf{b}_j$, $j\in J$, which are linear combinations of Heisenberg operators $\MOM_t,\POS_t$, $t\in\Delta_2(\partial X)$, such that the elements of the set $\{\mathsf{a}_i,\mathsf{b}_j\vert\ i\in I,j\in J\}$ are linearly independent and mutually commuting. Classically this set of operators correspond to a new real linear polarization of the cotangent bundle to the space $\bR^{\Delta_2(\partial X)}$.
Let $\vert\xi,\eta\rangle$, $\xi\in\REALS^{I},\eta\in\REALS^J$, be a normalized complete basis in $V(\partial X)$ where the operators $\mathsf{a}_i,\mathsf{b}_j$ act diagonally
\[
\mathsf{a}_i\vert\xi,\eta\rangle=\vert\xi,\eta\rangle\xi_i,\ i\in I,\quad \mathsf{b}_j\vert\xi,\eta\rangle=\vert\xi,\eta\rangle\eta_j,\ j\in J.
\]
Here $\vert\xi,\eta\rangle$ should be thought of
as the operator, which induces the isomorphism between the quantization in the original position coordinates introduced in Section \ref{qtt} and this new real linear polarization. This operator of course induces an isomorphism of the relevant Schwartz spaces.
Let $\langle\xi,\eta\vert\Psi(X)\rangle$ denote the unique distribution which corresponds to $\vert\Psi(X)\rangle$ under this isomorphism, i.e.
\[
\vert\Psi(X)\rangle=\int_{\REALS^{I\sqcup J}}d\xi d\eta\vert\xi,\eta\rangle\langle\xi,\eta\vert\Psi(X)\rangle.
\]
Equations~\eqref{eq:ann}  now imply that this new presentation of our partition function is of the form
\begin{equation}\label{nf}
\langle\xi,\eta\vert\Psi(X)\rangle=\psi_X(\eta)\prod_{i\in I}\delta(\xi_i),
\end{equation}
where $\psi_X\in\S'(\REALS^J)$. When $X$ is a disjoint union of finitely many shaped tetrahedra, the function $\psi_X$ is smooth and exponentially decaying at infinity.
Let a shaped 3-manifold $Y$ be obtained from $X$ by gluing along a set of triangles in $\partial X$, where $X$ is such that its partition function is a distribution of the form (\ref{nf}), and the function $\psi_X$ is smooth and exponentially decaying at infinity.
Let now $\mathsf{G}_Y$ be the distribution which implements the gluing on the boundary of $X$ to get $Y$ from $X$ in the quantization with respect to the new polarization, i.e. it is the conjugate under  $\vert\xi,\eta\rangle$ of the original gluing distribution.
Then the partition function of $Y$ is given by
\begin{equation}\label{y-from-x}
\vert\Psi(Y)\rangle=\int_{\REALS^{I\sqcup J}}d\xi d\eta\,\mathsf{G}_Y\vert\xi,\eta\rangle\langle\xi,\eta\vert\Psi(X)\rangle.
\end{equation}
The only case when the integrand, i.e. the product of these two distributions, is ill defined, is if their Wave front sets violate the condition of Theorem \ref{mult}. Using the fact that $\mathsf{G}_Y$ is the distribution which implements the gluing, and the linearity of the change of polarizations, we see that this is the case if and only if the temperate distribution $\mathsf{G}_Y$ has the form
\begin{equation}\label{forbid}
\mathsf{G}_Y\vert\xi,\eta\rangle=\vert\chi(\xi,\eta)\rangle\delta\left(\sum_{i\in I}n_i\xi_i\right)
\end{equation}
where $n_i\in\INTEGERS$, $i\in I$ and $\vert\chi(\xi,\eta)\rangle$ is some other distribution which can be multiplied with this delta-function. Let us see that in the case of admissible $Y$ such a form of $\mathsf{G}_Y$ is impossible. Indeed, equation~\eqref{forbid} means that
\[
\sum_{i\in I}n_i\mathsf{G}_Y\mathsf{a}_i=0.
\]
Thinking of $\mathsf{G}_Y$ as a partition function of a shaped 3-manifold $M$, we conclude
that the homology class of $\sum_{i\in I}n_i\gamma_i$ is trivial in $H_1(M\setminus\Delta_0(M),\bZ)$, i.e there exists  a 2-chain $C$ in $M$ such that
\[
\partial C=\sum_{i\in I}n_i\gamma_i.
\]
On the other hand, as $\gamma_i=\partial D_i$ in $X$, we come to the conclusion that
\[
C-\sum_{i\in I}n_iD_i
\]
is a nontrivial 2-cycle in $Y\setminus\Delta_0(Y)$ coming to contradiction with admissibility of $Y$. Thus, in the case when $H_2(Y\setminus\Delta_0(M),\bZ)=0$, the
integrand in \eqref{y-from-x} is a well defined temperate distribution. Furthermore, due to the decay properties of $\psi_X$, this integrand has also the necessary extension properties to be pushed forward by Proposition~\ref{pf} to a temperate distribution. The partition function of $Y$ is thus a well defined temperate distribution.
\end{proof}

\subsection{Proof of Theorem~\ref{Main}}
We define the value of the functor  on leveled shaped tetrahedra by formulae~\eqref{fh} and \eqref{tet-int-f}, or equivalently by \eqref{tet-fun}.
The value of $F_\hbar$ on any admissible leveled shaped triangulated pseudo 3-manifold $(X,\ell_X)$ is calculated by composing the morphisms associated with constituent tetrahedra. Theorem~\ref{convergence} implies that the result of composition is well defined, while the symmetry properties given by Fundamental Lemma~\ref{fund-lem} and explained in Section~\ref{tqft-rules} ensure that the obtained morphism is independent of the choice of vertex ordering of tetrahedra. Furthermore, Propositions~\ref{gauge-trans} and \ref{charged-pentagon} imply that the morphism is gauge invariant and does not change under Pachner refinements, in other words, $F_\hbar(X,\ell_X)$ depends on only the equivalence class of $(X,\ell_X)$. Finally, $F_\hbar$ obviously respects the composition, because a composed morphism eventually again is reduced to composition of its constituent tetrahedra.

\section{Examples of calculation}\label{ex}
In the following examples we encode an oriented triangulated pseudo 3-manifold $X$ into a diagram where a tetrahedron $T$ is represented by an element
 \begin{center}
 \begin{tikzpicture}[scale=.3]
 \draw[very thick] (0,0)--(3,0);
 \draw (0,0)--(0,1);\draw (1,0)--(1,1);\draw (2,0)--(2,1);\draw (3,0)--(3,1);
 \end{tikzpicture}
 \end{center}
 where the vertical segments, ordered from left to right, correspond to the faces $\partial_0T,\partial_1T,\partial_2T,\partial_3T$ respectively. When we glue tetrahedron along faces, we illustrate this by joining the corresponding vertical segments.

\subsection{The complement of the trefoil knot}
Let $X$ be represented by the diagram
\begin{equation}\label{P:trefoil}
\begin{tikzpicture}[baseline=5pt,scale=.3]
\draw[very thick] (0,0)--(3,0);\draw[very thick] (0,1)--(3,1);
\draw (0,0)--(0,1);
\draw (1,0)--(1,1);
\draw (2,0)--(2,1);
\draw (3,0)--(3,1);
\end{tikzpicture}
\end{equation}
Thus, choosing an appropriate orientation, it consists of two positive tetrahedra
$T_1$ and $T_2$ with four identifications
\[
\partial_iT_1\simeq\partial_{3-i}T_2,\quad i\in\{0,1,2,3\},
\]
so that $\partial X=\varnothing$. Hence, combinatorially, we have $\Delta_0(X)=\{*\}$, $\Delta_1(X)=\{e_0,e_1\}$, $\Delta_2(X)=\{f_0,f_1,f_2,f_3\}$, and $\Delta_3(X)=\{T_1,T_2\}$ with the boundary maps
\[
f_i=\partial_iT_1=\partial_{3-i}T_2,\quad i\in\{0,1,2,3\},
\]
\[
\partial_if_j=\left\{\begin{array}{cl}
e_0,&\mathrm{if}\ i=1\ \mathrm{and}\ j\in\{1,2\};\\
e_1,&\mathrm{otherwise},
\end{array}
\right.
\]
\[
\partial_i e_j=*,\quad i,j\in\{0,1\}.
\]
The topological space $X\setminus\{*\}$ is homeomorphic to the complement of the trefoil knot. The set $\Delta_{3,1}(X)$ consists of the elements $(T_i, e_{j,k})$ for $i\in\{1,2\}$ and $0\le j<k\le 3$. We fix a shape structure
\[
\alpha_X\colon \Delta_{3,1}(X)\to \REALS_{>0}
\]
by the formulae
\[
\alpha_X(T_i, e_{0,1})=2\pi a_i,\quad \alpha_X(T_i, e_{0,2})=2\pi b_i,\quad\alpha_X(T_i, e_{0,3})=2\pi c_i,\quad i\in\{1,2\},
\]
where $a_i+b_i+c_i=\frac12$. The weight function
\[
\omega_X\colon \Delta_1(X)\to\REALS_{>0}
\]
 takes the values
\[
\omega_X(e_0)=2\pi(c_1+c_2)=:2\pi w,\quad \omega_X(e_1)=2\pi(2-c_1-c_2)=2\pi(2-w).
\]
As the trefoil knot is not hyperbolic, the completely balanced case $w=1$ is not accessible directly, but it can be approached arbitrarily closely in the limit
\[
c_i\to 1/2,\ a_i\to 0,\quad i\in\{1,2\}.
\]
Denoting
\[
\nu_{x,y}:=e^{\imun4\pi\cla^2 x(x+y)}\nu(x-y),\quad \nu(x):=e^{-\imun\pi\cla^2(4x+1)/6},
\]
we calculate the partition function
\begin{multline*}
\parfun(X)=
\int_{\REALS^4}dx_0dx_1dx_2dx_3\,\langle x_0,x_2\vert\PTOLEMY(a_1,c_1)\vert x_1,x_3\rangle
\langle x_3,x_1\vert\PTOLEMY(a_2,c_2)\vert x_2,x_0\rangle\\
=e^{-\imun\pi/6}\int_{\REALS^4}dx_0dx_1dx_2dx_3\,\delta(x_0+x_2-x_1)\delta(x_3+x_1-x_2)\\
\times
\cff_{c_1,b_1}(x_3-x_2)\cff_{c_2,b_2}(x_0-x_1)e^{\imun2\pi( x_0(x_3-x_2)+x_3(x_0-x_1)}\\
=e^{-\imun\pi/6}\int_{\REALS^2}dxdy\,\cff_{c_1,b_1}(x)\cff_{c_2,b_2}(x+y)
e^{-\imun2\pi y^2}\\
=e^{-\imun\pi/6}\nu_{c_1,b_1}\nu_{c_2,b_2}
\int_{\REALS^2}dxdy\,\frac{e^{-\imun4\pi\cla(x w+yc_2)-\imun2\pi y^2}}{\QDILOG(x+2\cla a_1-\cla)\QDILOG(x+y+2\cla a_2-\cla)}\\
=e^{-\imun\pi/6}\nu_{c_1,b_1}\nu_{c_2,b_2}
\int_{\REALS^2}dxdy\,\frac{e^{-\imun4\pi\cla((x-2\cla a_1) w+yc_2)-\imun2\pi y^2}}{\QDILOG(x-\cla)\QDILOG(x+y+2\cla (a_2-a_1)-\cla)}\\
=e^{\imun8\pi\cla^2wa_1-\imun\pi/6}\nu_{c_1,b_1}\nu_{c_2,b_2}\\
\times
\int_{\REALS}dy\,e^{-\imun4\pi\cla yc_2-\imun2\pi y^2}\IHG{2}(-\infty,-\infty;y+2\cla (a_2-a_1);\cla(1-2w))
\end{multline*}
In the fully balanced limit $w\to1$, the last integrand can be calculated explicitly by the integral analogue of Saalsch\"utz summation formula
\[
\IHG{2}(-\infty,-\infty;d;-\cla)=\zeta_o^3
e^{\IMUN\pi d(2\cla-d)},
\]
so that we obtain
\[
\lim_{w\to1}\parfun(X)=e^{-\imun\pi/6}\nu_{1/2,0}^2\zeta_o^3\int_{\REALS}dye^{-\imun3\pi y^2}=\frac{e^{-\imun\pi/6}}{\sqrt{3}}.
\]
It is interesting to note that the result is independent of the quantization parameter $\hbar$.
\subsection{One vertex H-triangulation of $(S^3,3_1)$}
Let $X$ be represented by the diagram
 \[
 \begin{tikzpicture}[scale=.3]
 \draw[very thick] (0,0)--(3,0);
 \draw(1,0)..controls (1,1/2) and (2,1/2) ..(2,0);
 \draw(0,0)..controls (0,1) and (3,1) ..(3,0);
 \end{tikzpicture}
\]
where the trefoil knot is represented by the edge connecting the minimal and the maximal vertices of the tetrahedron.
Choosing an orientation, it consists of one positive tetrahedron
 $T$ with two identifications
\[
\partial_{i}T\simeq\partial_{3-i}T,\quad i\in\{0,1\},
\]
so that $\partial X=\varnothing$, and as a topological space, $X$ is homeomorphic to 3-sphere. Combinatorially, we have $\Delta_0(X)=\{*\}$, $\Delta_1(X)=\{e_0,e_1\}$, $\Delta_2(X)=\{f_0,f_1\}$, and $\Delta_3(X)=\{T\}$ with the boundary maps
\[
f_{i}=\partial_{i}T=\partial_{3-i}T,\quad i\in\{0,1\},
\]
\[
\partial_if_j=\left\{\begin{array}{cl}
e_0,&\mathrm{if}\ i=j=1;\\
e_1,&\mathrm{otherwise},
\end{array}
\right.
\]
\[
\partial_i e_j=*,\quad i,j\in\{0,1\}.
\]
 The set $\Delta_{3,1}(X)$ consists of elements $(T, e_{j,k})$ for $0\le j<k\le 3$. We fix a shape structure
\[
\alpha_X\colon \Delta_{3,1}(X)\to \REALS_{>0}
\]
by the formulae
\[
\alpha_X(T, e_{0,1})=2\pi a,\quad \alpha_X(T, e_{0,2})=2\pi b,\quad\alpha_X(T, e_{0,3})=2\pi c,
\]
where $a+b+c=\frac12$. The weight function
\[
\omega_X\colon \Delta_1(X)\to\REALS_{>0}
\]
 takes the values
\[
\omega_X(e_0)=2\pi c,\quad \omega_X(e_1)=2\pi(1-c).
\]
Geometrically, the most interesting case corresponds to balanced edge $e_1$ with $c=0$, since $e_0$ is knotted as the trefoil knot. However, this point is not directly accessible but it can be approached arbitrarily closely.
We calculate the partition function
\begin{multline*}
\parfun(X)=
\int_{\REALS^2}dx_0dx_1\,\langle x_0,x_1\vert\PTOLEMY(a,c)\vert x_1,x_0\rangle\\
=e^{-\frac{\imun\pi}{12}}\int_{\REALS^2}dx_0dx_1\,\delta(x_0)
\cff_{c,b}(x_0-x_1)e^{\imun2\pi x_0(x_0-x_1)}\\
=e^{-\frac{\imun\pi}{12}}\int_{\REALS}dx\,
\cff_{c,b}(x)=e^{-\frac{\imun\pi}{6}}
\cff_{b,a}(0)=\frac{e^{-\frac{\imun\pi}{6}}\nu_{b,a}}{\QDILOG(2\cla c-\cla)}.
\end{multline*}
Note, that the point $c=0$ is a simple pole. The renormalized partition function
\[
\tilde\parfun(X):=\lim_{c\to0} \QDILOG(2\cla c-\cla)\parfun(X)=\nu_{b,a}e^{-\frac{\imun\pi}{6}}=\frac{1}{\nu(b)}e^{-\frac{\imun\pi}{6}}
\]
is considered as a  non-compact analogue of the quantum dilogarithmic invariant of \cite{K4}.

\subsection{One tetrahedron with one face identification}
Let $X$ consists of one tetrahedron with one face identification
\[
 \begin{tikzpicture}[scale=.3]
 \draw[very thick] (0,0)--(3,0);
 \draw(2,0)--(2,1);
 \draw(3,0)--(3,1);
 \draw(0,0)..controls (0,1) and (1,1) ..(1,0);
 \end{tikzpicture}
\]
It has the only interior edge given by the edge connecting the maximal and the next to maximal vertices of the tetrahedron. This partition function can be used as the canonical part of a one vertex H-triangulation for a pair $(M,K)$ consisting  a 3-manifold $M$ and a knot $K$ in $M$, the only interior edge representing the knot.

Choosing the positive orientation, the corresponding operator is calculated as follows:
\begin{multline*}
\langle x\vert \parfun(X)\vert y\rangle=\int dz\,\langle z,x\vert\PTOLEMY(a,c)\vert z,y\rangle=
\delta(x)\tilde\cff'_{a,c}(y-x)\int_{\REALS}dz\,e^{\imun2\pi z(y-x)}\\
=\delta(x)\delta(y)\tilde\cff'_{a,c}(0)
=\delta(x)\delta(y)\frac{e^{-\imun\pi/12}\nu_{c,b}}{\QDILOG(2\cla a-\cla)}.
\end{multline*}
Here, the weight of the interior edge is  $2\pi a$.

\subsection{The complement of the figure--eight knot}
Let $X$ be represented by the diagram
\begin{equation}\label{E:8graph}
\begin{tikzpicture}[baseline=5pt,scale=.3]
\draw[very thick] (0,0)--(3,0);\draw[very thick] (0,1)--(3,1);
\draw(0,0)--(1,1);
\draw (1,0)--(0,1);
\draw (2,0)--(3,1);
\draw (3,0)--(2,1);
\end{tikzpicture}
\end{equation}
Again, choosing an orientation,
it consists of one positive tetrahedron $T_+$ and one negative tetrahedron $T_-$ with four identifications
\[
\partial_{2i+j}T_+\simeq\partial_{2-2i+j}T_-,\quad i,j\in\{0,1\},
\]
so that $\partial X=\varnothing$. Combinatorially, we have $\Delta_0(X)=\{*\}$, $\Delta_1(X)=\{e_0,e_1\}$, $\Delta_2(X)=\{f_0,f_1,f_2,f_3\}$, and $\Delta_3(X)=\{T_+,T_-\}$ with the boundary maps
\[
f_{2i+j}=\partial_{2i+j}T_+=\partial_{2-2i+j}T_-,\quad i,j\in\{0,1\},
\]
\[
\partial_if_j=\left\{\begin{array}{cl}
e_0,&\mathrm{if}\ j-i\in\{0,1\};\\
e_1,&\mathrm{otherwise},
\end{array}
\right.
\]
\[
\partial_i e_j=*,\quad i,j\in\{0,1\}.
\]
The topological space $X\setminus\{*\}$ is homeomorphic to the complement of the figure--eight knot. The set $\Delta_{3,1}(X)$ consists of elements $(T_\pm, e_{j,k})$ for $0\le j<k\le 3$. We fix a shape structure
\[
\alpha_X\colon \Delta_{3,1}(X)\to \REALS_{>0}
\]
by the formulae
\[
\alpha_X(T_\pm, e_{0,1})=2\pi a_\pm,\quad \alpha_X(T_\pm, e_{0,2})=2\pi b_\pm,\quad\alpha_X(T_\pm, e_{0,3})=2\pi c_\pm,
\]
where $a_\pm+b_\pm+c_\pm=\frac12$. The weight function
\[
\omega_X\colon \Delta_1(X)\to\REALS_{>0}
\]
 takes the values
\[
\omega_X(e_0)=2\pi(2a_++c_++2b_-+c_-)=:2\pi w,\quad \omega_X(e_1)=2\pi(2-w).
\]
As the figure--eight knot is hyperbolic, the completely balanced case $w=1$ is accessible directly, the complete hyperbolic structure corresponding to the symmetric point
\[
a_\pm=b_\pm=c_\pm=\frac16.
\]
We calculate the partition function
\begin{multline*}
\parfun(X)=
\int_{\REALS^4}dx_0dx_1dx_2dx_3\,\langle x_0,x_2\vert\PTOLEMY(a_+,c_+)\vert x_1,x_3\rangle
\overline{\langle x_2,x_0\vert\PTOLEMY(a_-,c_-)\vert x_3,x_1\rangle}\\
=\int_{\REALS^4}dx_0dx_1dx_2dx_3\,\delta(x_0+x_2-x_1)\delta(x_2+x_0-x_3)\\
\times
\cff_{c_+,b_+}(x_3-x_2)\overline{\cff_{c_-,b_-}(x_1-x_0)}e^{\imun2\pi( x_0(x_3-x_2)-x_2(x_1-x_0))}\\
=\int_{\REALS^2}dxdy\,\cff_{c_+,b_+}(x)\overline{\cff_{c_-,b_-}(y)}
e^{\imun2\pi (x^2-y^2)}
=\varphi_{c_+,b_+}\overline{\varphi_{c_-,b_-}},
\end{multline*}
where
\[
\varphi_{c,b}:=\int_{\REALS}dz\,\cff_{c,b}(z)e^{\imun2\pi z^2}
=\mu_{c,b}\varphi(2b+c),
\]
\[
\mu_{c,b}:=\nu_{c,b}e^{\imun 8\pi\cla^2b(b+c)},
\]
\[
\varphi(x):=\int_{\REALS-\imun d}dz\, \frac{e^{\imun 2\pi z^2+\imun 4\pi\cla zx}}{\QDILOG(z)}
\]
where $d\in\bR$ is chosen so that the integral converges absolutely.
The condition $w=1$ corresponds to fully balanced case which implies that
\[
\lambda:= 2b_++c_+=2b_-+c_-
\]
so that in this case the partition function takes the form
\[
\frac{\mu_{c_-,b_-}}{\mu_{c_+,b_+}}\parfun(X)=|\varphi(\lambda)|^2=
\int_{\REALS+\imun0}dx\,\chi_{4_1}(x,\lambda),
\]
where
\begin{equation}\label{eq:chi}
\chi_{4_1}(x,\lambda):=\chi_{4_1}(x)e^{\imun4\pi\cla x\lambda},\quad \chi_{4_1}(x):=\int_{\REALS-\imun0}dy\,\frac{\QDILOG(x-y)}{\QDILOG(y)}e^{\imun2\pi x(2y-x)}.
\end{equation}
It would be interesting to tie these computations up with the refined asymptotics given in Theorem 1 of \cite{AH}.

\subsection{One-vertex H-triangulation of $(S^3,4_1)$}
Let $X$ be given by the diagram
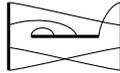
\[
 \begin{tikzpicture}[scale=.3]
 \draw[very thick] (0,0)--(0,3);
 \draw[very thick] (1,3/2)--(4,3/2);
 \draw[very thick] (5,0)--(5,3);
 \draw(1,3/2)..controls (1,2) and (2,2)..(2,3/2);
 \draw(3,3/2)..controls (3,2) and (3/2,5/2)..(0,2);
 \draw(4,3/2)..controls (4,2) and (4.5,3)..(5,3);
 \draw(0,3)..controls (1/2,3) and (4.5,2)..(5,2);
 \draw(0,1)..controls (1/2,1) and (4.5,0)..(5,0);
 \draw(0,0)..controls (1/2,0) and (4.5,1)..(5,1);
 \end{tikzpicture}
\]
where the figure-eight knot is represented by the edge of the central tetrahedron connecting the maximal and the next to maximal vertices.
Choosing positive central tetrahedron, the left tetrahedron will be positive and the right one negative. For technical reasons, we impose  the following condition on the shape structure: $2b_++c_+=2b_-+c_-=:\lambda$, and consider the following function
\begin{multline*}
f_X(x):=\int_{\REALS^{3}}dydudv\,\langle y,u\vert\PTOLEMY(a_+,c_+)\vert x,v\rangle
\overline{\langle u,y\vert\PTOLEMY(a_-,c_-)\vert v,x\rangle}\\
=\int_{\REALS^{3}}dydudv\,\delta(y+u-x)\delta(u+y-v)\tilde\cff'_{a_+,c_+}(v-u)
\overline{\tilde\cff'_{a_-,c_-}(x-y)}
e^{\imun2\pi (y(v-u)-u(x-y))}\\
=\int_{\REALS}dy\,\tilde\cff'_{a_+,c_+}(y)
\overline{\tilde\cff'_{a_-,c_-}(x-y)}
e^{\imun2\pi (y^2-(x-y)^2)}\\
=\frac{\nu_{c_+,b_+}}{\nu_{c_-,b_-}}
\int_{\REALS}dy\,\frac{\QDILOG(x-y-2\cla a_-+\cla)}{\QDILOG(y+2\cla a_+-\cla)}e^{\imun2\pi x(2y-x)-\imun 4\pi\cla (y c_++(x-y)c_-)}\\
=\frac{\nu_{c_+,b_+}}{\nu_{c_-,b_-}}e^{\imun2\pi\cla^2(c_-^2-c_+^2)}
\chi_{4_1}(x+\cla(c_--c_+),\lambda),
\end{multline*}
where the function $\chi_{4_1}(x,\lambda)$ is defined in \eqref{eq:chi}.
The partition function has the form
\[
\parfun(X)=f_X(0)\frac{e^{-\imun\pi/12}\nu_{c_0,b_0}}{\QDILOG(2\cla a_0-\cla)},
\]
where the weight on the knot is $2\pi a_0$.  In the limit when $a_0\to0$  the conditions $a_+=a_-$, $b_+=b_-$ imply that all edges except for the knot become balanced, and
the renormalized partition function takes the form
\[
\tilde\parfun(X):=\lim_{a_0\to0}\QDILOG(2\cla a_0-\cla)\parfun(X)=\frac{e^{-\imun\pi/12}}{\nu(c_0)}\chi_{4_1}(0).
\]
\subsection{The complement of the knot $5_2$}
Let $X$ be represented by the diagram
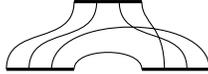
\[
 \begin{tikzpicture}[scale=.3]
 \draw[very thick] (0,0)--(3,0);
 \draw[very thick] (6,0)--(9,0);
 \draw[very thick] (3,3)--(6,3);
 \draw(3,0)..controls (3,1) and (6,1)..(6,0);
 \draw(0,0)..controls (0,2) and (3,1)..(3,3);
 \draw(1,0)..controls (1,2) and (4,1)..(4,3);
 \draw(2,0)..controls (2,2.5) and (9,2.5)..(9,0);
 \draw(5,3)..controls (5,1) and (8,2)..(8,0);
 \draw(6,3)..controls (6,2) and (7,2)..(7,0);
 \end{tikzpicture}
\]
We denote $T_1,T_2,T_3$ the left, right, and top tetrahedra respectively. We choose the orientation so that  all tetrahedra are positive.
We impose the conditions that all edges are balanced which correspond to two equations
\[
2a_3=a_1+c_2,\quad b_3=c_1+b_2.
\]
Let us denote
\[
\cff_{a,c}(x,y):=\tilde\cff'_{a,c}(y)e^{\imun2\pi xy}.
\]
The partition function has the form
\(
\parfun(X)=\int_{\REALS}dx\,f_X(x),
\)
where
\begin{multline*}
f_X(x):=\\\int_{\REALS^5}dydzdudvdw\,
 \langle z,w\vert\PTOLEMY(a_1,c_1)\vert u,x\rangle\langle x,v\vert\PTOLEMY(a_2,c_2)\vert y,w\rangle\langle y,u\vert\PTOLEMY(a_3,c_3)\vert v,z\rangle\\
=\int_{\REALS^5}dydzdudvdw\,  \delta(z+w-u)\delta(x+v-y)\delta(y+u-v)\cff_{a_1,c_1}(z,x-w)\cff_{a_2,c_2}(x,w-v)\\
\times\cff_{a_3,c_3}(y,z-u)=
\int_{\REALS^2}dydz\, \cff_{a_1,c_1}(z,z+2x)\cff_{a_2,c_2}(x,-y-z)\cff_{a_3,c_3}(y,x+z)\\
=
\int_{\REALS^2}dydz\, \cff_{a_1,c_1}(z-x,z+x)\cff_{a_2,c_2}(x,x-y-z)\cff_{a_3,c_3}(y,z)\\
=
\int_{\REALS^2}dydz\, \cff_{a_1,c_1}(z-x,z+x)\cff_{a_2,c_2}(x,y)\cff_{a_3,c_3}(x-y-z,z)\\
=e^{-\imun\pi/4}
\int_{\REALS^2}dydz\, \cff_{c_1,b_1}(z+x)\cff_{c_2,b_2}(y)\cff_{c_3,b_3}(z)
e^{\imun2\pi(x-y)(z-x))}\\
=e^{-\imun\pi/4}
\int_{\REALS}dz\, \cff_{c_1,b_1}(z+x)\tilde\cff_{c_2,b_2}(z-x)\cff_{c_3,b_3}(z)
e^{\imun2\pi (z-x)x}\\
=e^{-\imun\pi/3}
\int_{\REALS}dz\, \cff_{c_1,b_1}(z+x)\cff_{b_2,a_2}(z-x)\cff_{c_3,b_3}(z)
e^{\imun\pi (z-x)(z+x)}
=e^{-\imun\pi/3}\nu'_{c_1,b_1}\nu'_{b_2,a_2}\\
\times\nu'_{c_3,b_3}\int_{\REALS-\imun0}dz\, \frac{
e^{\imun\pi(z-x'+\cla(1-2c_2)) (z+x'+\cla(1-2a_1))-\imun4\pi\cla(c_1(z+x')+b_2(z-x')+c_3z)}}{\QDILOG(z+x')\QDILOG(z-x')\QDILOG(z)}\\
=\nu'_{c_1,b_1}\nu'_{b_2,a_2}\nu'_{c_3,b_3}e^{\imun\pi\cla^2(1-2a_1)(1-2c_2)}
\chi_{5_2}(x',a_1-c_1+b_2-a_3)
\end{multline*}
where
\[
x':=x+2\cla(a_1-a_3),
\]
\begin{equation}\label{eq:chi5_2}
\chi_{5_2}(x,\lambda):=\chi_{5_2}(x)e^{\imun4\pi\cla x\lambda},\quad \chi_{5_2}(x):=e^{-\imun\pi/3}\int_{\REALS-\imun0}dz\, \frac{
e^{\imun\pi(z-x) (z+x)}}{\QDILOG(z+x)\QDILOG(z-x)\QDILOG(z)}.
\end{equation}

\subsection{One-vertex H-triangulation of $(S^3,5_2)$}
Let $X$ be represented by the diagram
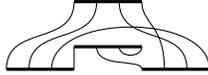
\[
 \begin{tikzpicture}[scale=.3]
 \draw[very thick] (0,0)--(3,0);
 \draw[very thick] (3,1)--(6,1);
 \draw[very thick] (6,0)--(9,0);
 \draw[very thick] (3,3)--(6,3);
 \draw(5,1)..controls (5,1/2) and (6,1/2)..(6,1);
 \draw (3,0)--(3,1);
 \draw(4,1)..controls (4,0) and (6,.5)..(6,0);
 \draw(0,0)..controls (0,2) and (3,1)..(3,3);
 \draw(1,0)..controls (1,2) and (4,1)..(4,3);
 \draw(2,0)..controls (2,2.5) and (9,2.5)..(9,0);
 \draw(5,3)..controls (5,1) and (8,2)..(8,0);
 \draw(6,3)..controls (6,2) and (7,2)..(7,0);
 \end{tikzpicture}
\]
We denote $T_0,T_1,T_2,T_3$ the central, left, right, and top tetrahedra respectively. If we choose the orientation so that the central tetrahedron $T_0$ is negative then all other tetrahedra are positive.
The edge representing the knot $5_2$ connects the last two edges of  $T_0$, so that the weight on the knot is given by $2\pi a_0$.  In the limit $a_0\to0$,  all edges, except for the knot, become balanced under the conditions
\[
a_1=c_2=a_3,\quad b_3=c_1+b_2,
\]
and the renormalized partition function takes the form
\[
\tilde\parfun(X):=\lim_{a_0\to0}\QDILOG(2\cla a_0-\cla)\parfun(X)=\frac{e^{\imun\pi/4}}{\nu(c_0)}\chi_{5_2}(0).
\]
\section{Proof of Theorem~\ref{th:4-1--5-2}}\label{proof-th3}
 In the case of the pairs $(S^3,4_1)$ and $(S^3,5_2)$, the first two parts of Conjecture~\ref{conj} directly follow from the results of calculations in Section~\ref{ex}, so that in order to complete the proof of Theorem~\ref{th:4-1--5-2}, we only need to prove part (3) of Conjecture~\ref{conj}.

 For $n>1$ and $\hbar>0$, consider the following function
\[
g_n(\hbar)=\frac1{2\pi\sqrt{\hbar}}\int_{\bR-\imun 0}f_n(\hbar,z)dz,\quad
f_n(\hbar,z):=\QDILOG\left(\frac{z}{2\pi\sqrt{\hbar}}\right)^{-n}e^{\frac{\imun z^2}{4\pi\hbar}}.
\]
From formulae~\eqref{eq:chi} and \eqref{eq:chi5_2} it follows that
\[
\chi_{4_1}(0)=g_2(\hbar),\quad \chi_{5_2}(0)=g_3(\hbar).
\]
Thus, we need to calculate the asymptotic behavior of $g_n(\hbar)$ at small $\hbar$ and then to take the particular cases with $n=2$ and $n=3$.

We denote
\[
h_n(\hbar,z):=\frac{\partial \log f_n(\hbar,z)}{\partial z}
\]
Let $S_{n,\hbar}$ be the set of solutions of the equation $h_n(\hbar,z)=0$, and let $z_{n,\hbar}\in S_{n,\hbar}$ be such that $|f_n(\hbar,z_{n,\hbar})|=\min |f_n(\hbar,S_{n,\hbar})|$. We define a set
\[
C_{n,\hbar}:=\{z\in\bC\,\vert\ \arg f_n(\hbar,z)=\arg f_n(\hbar,z_{n,\hbar}),\ |f_n(\hbar,z)|\le|f_n(\hbar,z_{n,\hbar})|\}
\]
which is smooth one-dimensional contour which does not contain any other elements of $S_{n,\hbar}$.
\begin{lemma}
For sufficiently small $\hbar$, there exists an orientation on $C_{n,\hbar}$ such that
one has the equality
\[
2\pi\sqrt{\hbar}g_n(\hbar)=\int_{C_{n,\hbar}}f_n(\hbar,z)dz,
\]
and there exist a point $z_n$ in a strip around the real axis and  a smooth contour $C_n$ passing through $z_n$ such that $z_n=\lim_{\hbar\to0}z_{n,\hbar}$ and $C_n=\lim_{\hbar\to0}C_{n,\hbar}$.
\end{lemma}
\begin{proof}
By using the asymptotic formula~\eqref{eq:q-cl}, and the relation $\hbar=(\la+\la^{-1})^{-2}$, we have
\[
\QDILOG\left(\frac{x}{2\pi\sqrt{\hbar}}\right)
=e^{\frac1{2\pi\imun\hbar}\dil(-e^x)}g(x)\left(1+\mathcal{O}(\hbar)\right),
\]
where
\[
g(x):=\left(1+e^x\right)^{\frac{\imun x}{2\pi}}e^{\frac{\imun}{\pi}\dil(-e^x)}.
\]
Here, the Euler dilogarithm $\dil(-e^z)$ is extended analytically to the entire complex plane cut along two half lines $\re z=0, |\im z|>\pi$, so that the following identity is satisfied:
\[
\dil(-e^z)+\dil(-e^{-z})=-\frac12z^2-\frac{\pi^2}6.
\]
These formulae imply the following asymptotic formula for the integrand
\begin{equation}\label{eq:fn}
f_n(\hbar,z)
=e^{\frac{v_n(z)}{2\pi\imun\hbar}}g(z)^{-n}\left(1+\mathcal{O}(\hbar)\right)
\end{equation}
where
\[
v_n(z):=-n\dil(-e^z)-\frac12z^2.
\]
Taking the logarithm of equation~\eqref{eq:fn}, we obtain
\[
\log f_n(\hbar,z)=\frac1{2\pi\imun\hbar}(v_n(z)+\mathcal{O}(\hbar)),
\]
and, differentiating the latter with respect to $z$,
\[
h_n(\hbar,z)=\frac1{2\pi\imun\hbar}(v_n'(z)+\mathcal{O}(\hbar)),
\]
where
\begin{equation}\label{eq:crit}
v'_n(z):=\frac{\partial v_n(z)}{\partial z}=n\log(1+e^z)-z.
\end{equation}
Thus, for sufficiently small $\hbar$, the set $S_{n,\hbar}$ approaches the finite set $S_n$ of solutions of the equation $v'_n(z)=0$ which is entirely contained in a strip around the real axis. In particular, there exists $z_n\in S_n$ such that
\[
\lim_{\hbar\to0}z_{n,\hbar}=z_n.
\]
From the definition of $z_{n,\hbar}$, it is clear that $z_n$ is characterized by the property that it minimizes the imaginary part of $v_n(z)$ on the set $S_n$:
\[
\im v_n(z_n)=\min\im v_n(S_n).
\]
Also, by observing that
\[
\arg f_n(\hbar,z)=\im \log f_n(\hbar,z)=\frac1{2\pi\hbar}(-\re v_n(z)+\mathcal{O}(\hbar))
\]
we conclude that the contour $C_{n,\hbar}$, when $\hbar\to0$, approaches the set
\[
C_n:=\{z\in\bC\,\vert \re v_n(z)=\re v_n(z_n),\ \im v_n(z)\le\im v_n(z_n)\}.
\]
 It is a smooth contour which approaches  the asymptote $\re z + \im z=0$ for $\re z\to+\infty$, and the asymptote  $\re z -\im z=0$ for $\re z\to-\infty$ with the behavior for the imaginary part of $v_n(z)$:
\[
\lim_{z\in C_n,\re z\to\pm\infty}\im v_n(z)=-\infty
 \]
 The contour $C_n$ is schematically presented in this picture
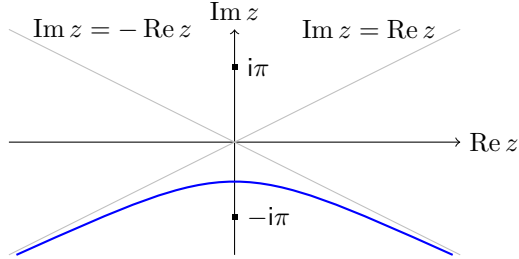
\[
 \begin{tikzpicture}[scale=1]
 \draw[->] (-3,0)--(3,0);
 \draw[->] (0,-1.5)--(0,1.5) ;
 \draw (3,0) node[right]{$\re z$};
 \draw (0,1.5) node[above]{$\im z$};
 \draw (-2.8,1.5) node[right]{$\im z=-\re z$};
 \draw (2.8,1.5) node[left]{$\im z=\re z$};
 \node [fill=black,inner sep=1pt,label=right:$\imun\pi$] at (0,1){};
 \node [fill=black,inner sep=1pt,label=right:$-\imun\pi$] at (0,-1){};
 \draw[gray!50] (-3,1.5)--(3,-1.5);
 \draw[gray!50] (-3,-1.5)--(3,1.5);
 \draw[blue,thick] (-2.9,-1.5)..controls (0,-.2) and (0,-.2) .. (2.9,-1.5);
 \end{tikzpicture}
\]
Orienting $C_n$ from left to right, it is clear that for sufficiently small $\hbar$ we have the equalities
\[
2\pi\sqrt{\hbar}g_n(\hbar)=\int_{C_n}f_n(\hbar,z)dz=\int_{C_{n,\hbar}}f_n(\hbar,z)dz
\]
\end{proof}
Now, as by construction $C_n$ contains the only critical point of $v_n(z)$ at $z=z_n$, the steepest descent method gives the following asymptotic formula

\[
g_n(\hbar)=e^{\frac{v_n(z_n)}{2\pi\imun\hbar}}
\frac{g(z_n)^{-n}}{\sqrt{\imun v_n''(z_n)}}\left(1+\mathcal{O}(\hbar)\right),
\]
in particular,
\[
\lim_{\hbar\to0}2\pi\hbar\log|g_n(\hbar)|=\im v_n(z_n)
\]
Finally, we have the following two particular cases:
\[
\im v_2(z_2)=-\vol(S^3\setminus 4_1),\quad \im v_3(z_3)=-\vol(S^3\setminus 5_2).
\]
which exactly correspond to part (3) of Conjecture~\ref{conj}.

\section{Appendix A. Faddeev's quantum dilogarithm}
Faddeev's quantum dilogarithm $\QDILOG(z)$ is defined by the integral
   \begin{equation}\label{eq:ncqdl}
\QDILOG(z)\equiv\exp\left(
\int_{\IMUN 0-\infty}^{\IMUN 0+\infty}
\frac{e^{-\IMUN 2 zw}\, dw}{4\sinh(w\la)
\sinh(w/\la) w}\right)
    \end{equation}
in the strip $|\im z|<|\im\cla|$, where
\[
\cla\equiv\IMUN(\la+\la^{-1})/2
\]
Define also
\begin{equation}\label{eq:cons}
\zeta_{inv}\equiv e^{\IMUN\pi(1+2\cla^2)/6}=e^{\IMUN\pi\cla^2}
\zeta_o^2,\quad
\zeta_o\equiv e^{\IMUN\pi(1-4\cla^2)/12}
\end{equation}
When  $\im\la^2>0$, the integral can be calculated explicitly
\begin{equation}\label{eq:ratio}
\QDILOG(z)=
(e^{2\pi (z+\cla)\la};q^2)_\infty/
(e^{2\pi (z-\cla)\la^{-1}};\bar q^2)_\infty
\end{equation}
where
\[
q\equiv e^{\IMUN\pi\la^2},\quad\bar q\equiv e^{-\IMUN\pi\la^{-2}}
\]
Using symmetry properties
\[
\QDILOG(z)=\mathop{\Phi_{-\la}}(z)=\mathop{\Phi_{1/\la}}(z)
\]
we choose
\[
\re\la>0,\quad \im \la\ge0
\]
$\QDILOG(z)$ can be continued analytically in variable
$z$ to the entire complex plane as a meromorphic function with essential singularity at infinity and with the following characteristic properties
\begin{description}
\item[zeros and poles]
  \begin{equation}
    \label{eq:polzer}
(\QDILOG(z))^{\pm1}=0\ \Leftrightarrow \ z=
\mp(\cla+m\IMUN\la+n\IMUN\la^{-1}),\ m,n\in\INTEGERS_{\ge0}
  \end{equation}
\item[behavior at infinity] depending on the direction along which the limit is taken
\begin{equation}\label{eq:asymp}
\QDILOG(z)\bigg\vert_{|z|\to\infty} \approx\left\{
\begin{array}{ll}
1&|\arg z|>\frac{\pi}{2}+\arg \la\\
\zeta_{inv}^{-1}e^{\IMUN\pi z^2}&|\arg z|<\frac{\pi}{2}-\arg\la\\
\frac{(\bar q^2;\bar q^2)_\infty}{\Theta(\IMUN\la^{-1}z;-\la^{-2})}&
|\arg z-\frac\pi2|<\arg\la\\
\frac{\Theta(\IMUN\la z;\la^{2})}{(q^2; q^2)_\infty}&
|\arg z+\frac\pi2|<\arg\la
\end{array}\right.
\end{equation}
where
\[
\Theta(z;\tau)\equiv\sum_{n\in\INTEGERS}
e^{\IMUN\pi\tau n^2+\IMUN2\pi zn}, \quad\im\tau>0
\]
\item[inversion relation]
\begin{equation}\label{eq:inversion}
\QDILOG(z)\QDILOG(-z)=\zeta_{inv}^{-1}e^{\IMUN\pi z^2}
\end{equation}
\item[functional equations]
\begin{equation}\label{eq:shift}
\QDILOG(z-\IMUN\la^{\pm1}/2)=(1+e^{2\pi\la^{\pm1}z })
\QDILOG(z+\IMUN\la^{\pm1}/2)
\end{equation}
\item[unitarity] when  $\la$ is real or on the unit circle
\begin{equation}\label{eq:qdlunitarity}
(1-|\la|)\im\la=0\ \Rightarrow\
\overline{\QDILOG(z)}=1/\QDILOG(\bar z)
\end{equation}
\item[quantum pentagon identity]
\begin{equation}\label{eq:pent}
\QDILOG(\MOM)\QDILOG(\POS)=\QDILOG(\POS)\QDILOG(\MOM+\POS)
\QDILOG(\MOM)
\end{equation}
where selfadjoint operators $\MOM$ and $\POS$ in $L^2(\REALS)$
satisfy the commutation relation~\eqref{eq:hei-cr}.
\end{description}

\subsection{Integral analog of $_1\psi_1$-summation formula of Ramanujan}
\label{sec:eioa-aiae-_1ps}

Following \cite{FKV}, consider the following Fourier integral
\begin{equation}\label{ramanint}
\RAMAN(u,v,w)\equiv
\int_{\REALS}\frac{\QDILOG(x+u)}{\QDILOG(x+v)}e^{2\pi\IMUN wx}\, dx
\end{equation}
where
\begin{equation}\label{restrictions1}
\im(v+\CLA)>0,\quad\im(-u+\CLA)>0, \quad \im(v-u)<\im w<0
\end{equation}
Conditions~\eqref{restrictions1} can be relaxed by deforming the integration contour in the complex plane of the variable $x$, keeping asymptotic directions within the sectors
$\pm(|\arg x|-\pi/2)>\arg\la$. Enlarged in this way the domain of the  variables $u,v,w$ has the form
\begin{equation}\label{restrictions2}
|\arg (\IMUN z)|<\pi-\arg\la,\quad z\in\{w,v-u-w,u-v-2\cla\}
\end{equation}
Considering $\QDILOG(z)$ as a non-compact analog the function
$(x;q)_\infty$, definition~\eqref{ramanint} can be interpreted as an integral analog of $_1\psi_1$-summation formula of Ramanujan.

By the method of residues, \eqref{ramanint} can be calculated explicitly
\begin{gather}\label{ramanres1}
  \RAMAN(u,v,w)=
  \zeta_o\frac{\QDILOG(u-v-\CLA)\QDILOG(w+\CLA)}{\QDILOG(u-v+w-\CLA)}
  e^{-2\pi\IMUN w(v+\CLA)} \\\label{ramanres2}
  =\zeta_o^{-1}\frac{\QDILOG(v-u-w+\CLA)}{\QDILOG(v-u+\CLA)
\QDILOG(-w-\CLA)}
  e^{-2\pi\IMUN w(u-\CLA)}
\end{gather}
where the two expressions in the right hand side are related by the inversion relation~\eqref{eq:inversion}.
\subsection{Fourier transformation formulae}
\label{sec:oidi-idai-oodua}

Special cases of $\RAMAN(u,v,w)$ lead to the following Fourier transformation formulae  for Faddeev's quantum dilogarithm
\begin{multline}\label{fourier1}
  \lefteqn{\FQDILOG_+(w)\equiv \int_{\REALS}\QDILOG(x)e^{2\pi \IMUN
      wx}\,dx
    =\RAMAN(0,v,w)\vert_{v\to-\infty}}\\
  =\zeta_o^{-1}e^{2\pi \IMUN w\CLA}/
  \QDILOG(-w-\CLA)=\zeta_oe^{-\IMUN\pi w^2}
  \QDILOG(w+\CLA)
\end{multline}
and
\begin{multline}\label{fourier2}
  \lefteqn{\FQDILOG_-(w)\equiv \int_{\REALS}(\QDILOG(x))^{-1}e^{2\pi
      \IMUN wx}\,dx
    =\RAMAN(u,0,w)\vert_{u\to-\infty}}\\
  =\zeta_oe^{-2\pi\IMUN w\CLA} \QDILOG(w+\CLA)=
  \zeta_o^{-1}e^{\IMUN\pi w^2}/\QDILOG(-w-\CLA)
\end{multline}
The corresponding inverse transformations look as follows
\begin{equation}\label{finv}
(\QDILOG(x))^{\pm1}=\int_{\REALS}\FQDILOG_\pm(y)e^{-2\pi\IMUN xy}dy
\end{equation}
where the pole $y=0$ should be surrounded from below.

\subsection{Other integral identities}
\label{sec:eioa-oiaa}

Faddeev's quantum dilogarithm also satisfies other analogs of hypergeometric identities, see for example \cite{GR}. Following \cite{V,K5}, for any $n\ge 1$, we define
\begin{equation}\label{eq:ihg}
\IHG{n}(a_1,\ldots,a_n;b_1,\ldots,b_{n-1};w)
\equiv\int_\REALS dx\,
e^{\IMUN2\pi x(w-\cla)}\prod_{j=1}^n
\frac{\QDILOG(x+a_j)}{\QDILOG(x+b_j-\cla)}
\end{equation}
where $b_n=\IMUN0$,
\[
\im(b_j)>0,\quad \im(\cla-a_j)>0,\quad \sum_{j=1}^n\im(b_j-a_j-\cla)<
\im(w-\cla)<0
\]
The integral analog of $_1\psi_1$-summation formula of Ramanujan in this notation takes the form
\begin{equation}\label{eq:raman}
\Psi(u,v,w)=e^{-\imun2\pi w(v+\cla)}\IHG{1}(u-v-\cla;w+\cla),\quad
\IHG{1}(a;w)=\zeta_o
\frac{\QDILOG(a)
\QDILOG(w)}{\QDILOG(a+w-\cla)}
\end{equation}
Equivalently, we can write
\begin{multline}\label{eq:ramanbar}
\bar\Psi_1(a;w)\equiv\int_\REALS\frac{\QDILOG(x+\cla-\IMUN0)}
{\QDILOG(x+a)}
e^{-\IMUN2\pi x(w+\cla)}dx\\
=e^{\IMUN2\pi(a+\cla)(w+\cla)}\IHG{1}(-a;-w)
=\zeta_o^{-1}
\frac{\QDILOG(a+w+\cla)}{\QDILOG(a)
\QDILOG(w)}
\end{multline}
Using the integral $_1\psi_1$-summation formula of Ramanujan, we can obtain the integral analog of the transformation of Heine for  $_1\phi_2$
\[
\IHG{2}(a,b;c;w)=\frac{\QDILOG(a)}{\QDILOG(c-b)}\,
\IHG{2}(c-b,w;a+w;b)
\]
Using here the evident symmetry
\[
\IHG{2}(a,b;c;w)=\IHG{2}(b,a;c;w)
\]
we obtain the integral analog of Euler--Heine transformation
\begin{multline}\label{eq:eu-he}
\IHG{2}(a,b;c;w)\\
=\frac{\QDILOG(a)\QDILOG(b)\QDILOG(w)}
{\QDILOG(c-b)\QDILOG(c-a)\QDILOG(a+b+w-c)}\,
\IHG{2}(c-a,c-b;c;a+b+w-c)
\end{multline}
Performing the Fourier transformation on $w$ and using equation~\eqref{eq:raman}, we obtain the integral version of Saalsch\"utz summation formula:
\begin{multline}\label{eq:saal}
\IHG{3}(a,b,c;d,a+b+c-d-\cla;-\cla)
=\zeta_o^3
e^{\IMUN\pi d(2\cla-d)}\\
\times\frac{\QDILOG(a)\QDILOG(b)\QDILOG(c)
\QDILOG(a-d)\QDILOG(b-d)\QDILOG(c-d)}{\QDILOG(a+b-d-\cla)\QDILOG(b+c-d-\cla)
\QDILOG(c+a-d-\cla)}
\end{multline}
One particular case is obtained by taking the limit $c\to-\infty$:
\begin{equation}\label{eq:saal1}
\IHG{2}(a,b;d;-\cla)=\zeta_o^3
e^{\IMUN\pi d(2\cla-d)}\frac{\QDILOG(a)\QDILOG(b)
\QDILOG(a-d)\QDILOG(b-d)}{\QDILOG(a+b-d-\cla)}
\end{equation}

\subsection{Quasi-classical limit of Faddeev's quantum dilogarithm}
\label{sec:eaac-idaa}

We consider the asymptotic limit $\la\to0$.
\begin{proposition}
For $\la\to0$ and fixed $x$, one has the following asymptotic expansion
\begin{equation}
  \label{eq:as-ex}
 \ln\QDILOG\left(\frac x{2\pi\la}\right)=\sum_{n=0}^\infty
\left(2\pi\imun\la^2\right)^{2n-1}\frac{B_{2n}(1/2)}{(2n)!}
\frac{\partial^{2n}\dil(-e^x)}{\partial x^{2n}}
\end{equation}
where $B_{2n}(1/2)$ --- Bernoulli polynomials evaluated at $1/2$.
\end{proposition}
\begin{proof}
From one hand side, we can write formally
\[
\ln\QDILOG\left(\frac {x-\imun\pi\la^2}{2\pi\la}\right)-
\ln\QDILOG\left(\frac {x+\imun\pi\la^2}{2\pi\la}\right)=
-2\sinh(\imun\pi\la^2\partial/\partial x)
\ln\QDILOG\left(\frac x{2\pi\la}\right)
\]
On the other hand side,
\[
\ln(1+e^x)=\frac{\partial}{\partial x}\int_{-\infty}^x\ln(1+e^z)dz=
-\frac{\partial}{\partial x}\dil(-e^x)
\]
Substituting these expressions into the linearized functional equation~(\ref{eq:shift})
with positive exponent of $\la$, we obtain
 \begin{equation}
   \label{eq:as-ex-i}
  2\pi\imun\la^2 \ln\QDILOG\left(\frac x{2\pi\la}\right)=
\frac{\imun\pi\la^2\partial/\partial x}{
\sinh(\imun\pi\la^2\partial/\partial x)}\dil(-e^x)
 \end{equation}
Using the expansions
\[
\frac z {\sinh(z)}=\sum_{n=0}^\infty B_{2n}(1/2)\frac{(2z)^{2n}}{(2n)!}
\]
we obtain (\ref{eq:as-ex}).
\end{proof}
\begin{corollary}
For $\la\to0$, one has
\begin{equation}
  \label{eq:q-cl}
\QDILOG\left(\frac x{2\pi\la}\right)
=e^{\frac1{2\pi\imun\la^2}\dil(-e^x)}\left(1+\mathcal{O}(\la^2)\right)
\end{equation}
\end{corollary}

\section{Appendix B. Temperate distributions}
\label{dist}
Recall the definition of the Schwartz space $\S(\bR^n)$ and its dual space $\S'(\bR^n)$ consisting of temperate distributions.

\begin{definition}
We denote by $\S(\bR^n)$ the set of all $\phi\in C^\infty(\bR^n)$ such that
$$
\sup_{x\in \bR^n} | x^\beta \partial^\alpha\phi(x)| < \infty
$$
for all multi-indices $\alpha$ and $\beta$.
We denote by $\S'(\bR^n)$ the space of linear functionals which are continuous with respect to the these seminorms.
\end{definition}

We further recall that the Fourier transform takes $\S(\bR^n)$ to $\S(\bR^n)$ and dually takes $\S'(\bR^n)$ to $\S'(\bR^n)$ (see e.g. \cite{Hor1}).
According to theorem 8.2.4. in \cite{Hor1}, we have for any projection $\pi : \bR^n \ra\bR^m$ for $m<n$ a well defined pull back map
$$\pi^* : \S'(\bR^m) \ra \S'(\bR^n).$$

Let us now recall the definition of the Wave Front Set of a distribution. Let $Z_{\bR^n}$ be the zero section of $T^*(\bR^n)$.

\begin{definition}
For a temperate distribution $u\in \S'(\bR^n)$, we define its \emph{Wave Front Set} to be the following subset of the cotangent bundle of $\bR^n$
$$
\wf(u) = \{ (x,\xi)\in T^*(\bR^n)- Z_{\bR^n} |\ \xi \in \Sigma_x(u)\}
$$
where
$$
\Sigma_x(u) = \cap_{\phi\in C^\infty_x(\bR^n)}\Sigma(\phi u).
$$
Here
$$
C^\infty_x(\bR^n) = \{ \phi \in C^\infty_0(\bR^n) | \phi(x) \neq 0\}
$$
and $\Sigma(v)$ are all $\eta \in \bR^n -\{0\}$ having no conic neighborhood $V$ such that
$$
|\hat{v}(\xi)| \leq C_N (1+ |\xi|)^{-N}, \ N\in\bZ_{>0},\  \xi\in V.
$$

\end{definition}

\begin{lemma}\label{wfdens}
Suppose $u$ is a bounded density on a $C^\infty$ sub-manifold $Y$ of $\bR^n$, then $u\in \S'(\bR^n)$ and
$$\wf(u) = \{(x,\xi) \in T^*(\bR^n) | x\in \supp u\mbox{, } \xi\neq 0 \mbox{ and } \xi(T_xY) = 0 \}.$$
\end{lemma}

In particular if $\supp u = Y$, then we see that $\wf(u)$ is the co-normal bundle of $Y$.

\begin{definition}
Let $u$ and $v$ be temperate distributions on $\bR^n$. Then we define
$$\wf(u) \oplus \wf(v) = \{ (x,\xi_1+\xi_2) \in T^*(\bR^n) | (x,\xi_1)\in \wf(u)\mbox{, }  (x,\xi_2) \in \wf(v)\}.$$
\end{definition}

\begin{theorem}\label{mult}
Let $u$ and $v$ be temperate distributions on $\bR^n$. If
$$\wf(u) \oplus \wf(v) \cap Z_{n} = \emptyset,$$
then the product of $u$ and $v$ exists and $uv\in \S'(\bR^n)$.
\end{theorem}

This is the content of Theorem IX.45 in \cite{RS2}.

In order to be able to push forward certain temperate distributions along a projection map from $\bR^n$ to $\bR^m$, where $m< n$, we need to introduce
the space $\S'(\bR^n)_m$.

\begin{definition}

We denote by $\S(\bR^n)_m$ the set of all $\phi\in C^\infty(\bR^n)$ such that
$$
\sup_{x\in \bR^n} | x^\beta \partial^\alpha (\phi)(x)| < \infty
$$
for all multi-indices $\alpha$ and $\beta$ such that if $\alpha_i =0$ then $\beta_i=0$ for $n-m<i\leq n$.
We define $\S'(\bR^n)_m$ to be the continuous dual of $\S(\bR^n)_m$ with respect to these semi-norms.
\end{definition}

We observe that if $\pi : \bR^n \ra \bR^{n-m}$ is the projection onto the first $n-m$ coordinates, then $\pi^*(\S(\bR^{n-m})) \subset \S(\bR^n)_m$. This means we have a well defined push forward map
$$
\pi_* : \S'(\bR^n)_m \ra \S'(\bR^{n-m}).
$$

\begin{proposition}\label{pf}
Suppose $Y$ is a linear subspace in $\bR^n$, $u$ a density on $Y$ with exponential decay in all directions in $Y$. Suppose $\pi : \bR^n \ra \bR^m$ is a projection for some $m<n$. Then $u\in \S'(\bR^n)_m$ and $\pi_*(u)$ is a density on $\pi(Y)$ with exponential decay in all directions of the subspace $\pi(Y)\subset \bR^m$.
\end{proposition}

\proof
Since the density $u$ on $Y$ has exponential decay in all directions in $Y$, it is clear that $u\in \S'(\bR^n)_m$, hence we can consider $\pi_*(u) \in S'(\bR^n)$. We now consider
\begin{equation}\label{piY}
\pi |_{Y} : Y \ra \pi(Y).
\end{equation}
Since the density $u$ on $Y$ has exponential decay in all direction in $Y$, it is clear we can integrate the density $u$ over the fibers of (\ref{piY})
to obtain the density $(\pi|_{Y})_*(u)$ on $\pi(Y)$, which has exponential decay in all direction in $\pi(Y)$.
It is clear that the density $(\pi|_{Y})_*(u)$ on $\pi(Y)$ represents $\pi_*(u) \in S'(\bR^n)$.
\eproof

\section{Appendix C. Categroids}
\label{Categroids}

We need a notion which is slightly more general than categories to define our TQFT functor.

\begin{definition}
A \emph{Categroid} $\C$ consist of a family of objects $\OBJ({\C})$ and for any pair of objects $A,B$ from $ \OBJ({\C})$ a set $\mor_{\C}(A,B)$ such that
the following holds
\begin{description}
\item[A] For any three objects $A,B,C$ there is a subset $K^{\C}_{A,B,C} \subset \mor_\C(A,B)\times \mor_\C(B,C)$, called the composable morphisms and a \emph{composition} map
$$
\circ : K^{\C}_{A,B,C} \ra  \mor_\C(A,C).
$$
such that composition of composable morphisms is associative.
\item[B] For any object $A$ we have an identity morphism $1_A \in  \mor_\C(A,A)$ which is composable with any morphism $f\in \mor_\C(A,B)$ or $g\in \mor_\C(B,A)$ and we have the equation
$$
1_A\circ f = f \mbox{, and } g \circ 1_A = g.
$$
\end{description}
\end{definition}

We observe that $\CC_a$ is a categroid with $K^{\CC_a}_{A,B,C}$ for three objects $A,B,C \in \CC_a$ given by
\begin{eqnarray*}
K^{\B_a}_{A,B,C} &=& \{ (X_1,X_2)\in \mor_{\CC_a}(A,B)\times \mor_{\CC_a}(B,C) | \\
& & \phantom{OS} H_2(X_1\circ X_2 - \Delta_0(X_1\circ X_2 ), \bZ) =0 \},
\end{eqnarray*}
where $\circ$ referes to the composition in $\CC$.
\begin{definition}
For the Categroid $\D$ of temperate distributions, the set of composable morphisms for three finite sets $n,m,l$ consists of the following subset
\begin{multline*}
K_{n,m,l}  =  \{(A,B)\in \S'(\bR^{n\cup m})\times  \S'(\bR^{m\cup l}) \mid \\
(\wf(\pi_{n,m}^*(A))\oplus \wf(\pi_{m,l}^*(B))) \cap Z_{n,m,l} = \varnothing, \
 \pi_{n,m}^*(A)\pi_{m,l}^*(B) \in \S'(\bR^{n\sqcup m\sqcup l})_m\}.
\end{multline*}
\end{definition}

\end{document}